\documentclass{amsart}
\usepackage{amssymb}
\usepackage[matrix,arrow,curve]{xy}
\usepackage{verbatim}
\usepackage{graphics}
\usepackage{a4wide,array}
\usepackage[utf8]{inputenc}
\usepackage{tikz}
\newtheorem{theorem}{Theorem}[section]
\newtheorem{lemma}[theorem]{Lemma}
\newtheorem{remark}[theorem]{Remark}

\newtheorem{corollary}[theorem]{Corollary}

\newtheorem{proposition}[theorem]{Proposition}

\usepackage{hyperref}
\hypersetup{
    colorlinks,
    citecolor=black,
    filecolor=black,
    linkcolor=black,
    urlcolor=black
}
\def\dd{\mathrm{d}}

\def\eps{\varepsilon}

\def\Z{\mathbf{Z}}
\def\N{\mathbf{N}}

\def\Q{\mathbf{Q}}

\def\R{\mathbf{R}}

\def\into{\hookrightarrow}
\def\onto{\twoheadrightarrow}

\def\la{\lambda}

\def\KK{\mathcal{K}}

\def\PF{\mathcal{P}_{\mathrm{f}}}

\date{January 30, 2020.}

\begin{document}
\centerline{}

\title[Hausdorff metric]{Hausdorff metric between simplicial complexes}
\author[I.~Marin]{Ivan Marin}
\address{LAMFA, UMR CNRS 7352, Universit\'e de Picardie-Jules Verne, Amiens, France}
\email{ivan.marin@u-picardie.fr}
\medskip

\begin{abstract} 
We introduce a distance function between simplicial complexes and study several of its properties. 
\end{abstract}

\maketitle

\tableofcontents

\section{Introduction}

In \cite{SRV} we have shown that any simplicial complex can be (geometrically) realized as a closed subset of the space of discrete random variables. As a consequence, a natural metric on the collection of all simplicial complexes is given by the Hausdorff distance between the associated sets. In this paper, we explore the corresponding metric structure on simplicial complexes.

It turns out that this metric structure can be described directly at the level of the usual `geometric realization' of the simplicial complexes, provided we endow them with the appropriate $L^1$ metric structure and the associated Hausdorff distance function (see corollary \ref{cor:dLKdMK}). Let $\mathcal{S}$ denote the collection of all simplicial complexes with vertices inside a given (infinite) set $S$, endowed with the distance function $d$. The main results we obtain are the following ones. \begin{itemize} 
\item If $\mathcal{K}_1,\mathcal{K}_2$ are finite simplicial complexes, then $d(\mathcal{K}_1,\mathcal{K}_2) \in \Q$. 
\item The distance function takes all possible values inside $[0,1]$ 
\item If the $(n+1)$-squelettons of $\mathcal{K}_1$ and $\mathcal{K}_2$ differ, then $d(\mathcal{K}_1,\mathcal{K}_2) \geq \frac{1}{n+2}$. 
\item $\mathcal{S}$ is \emph{not} locally compact. The union of all the finite-dimensional complexes is \emph{not} dense inside $\mathcal{S}$, and is equal to the collection of all isolated points of $\mathcal{S}$. \end{itemize} 
This distance function enables one to define a metric on the collection of isomorphism classes of simplicial complexes, when the cardinality of
the vertices is not greater than a given cardinal $\mathfrak{a}$. Taking for $S$ the corresponding ordinal, if $K$ is such an isomorphism class of simplicial complexes it admits a representative $\mathcal{K} \in \mathcal{S}$, and, for two such isomorphism classes $K_1,K_2$ and representatives $\mathcal{K}_1$, $\mathcal{K}_2$, the following abstract distance is well-defined $$ d(K_1,K_2) = \inf_{\sigma \in \mathfrak{S}(S)} d(\sigma(\mathcal{K}_1),\mathcal{K}_2) = \inf_{\sigma \in \mathcal{S}} d(\mathcal{K}_1,\sigma(\mathcal{K}_2)) $$ where $\mathfrak{S}(S)$ denotes the group of permutations of $S$. This distance function does not depend on the chosen cardinal $\mathfrak{a}$ (see proposition \ref{prop:indepST}). We show that
\begin{itemize} 
\item If $K_1$, $K_2$ are isomorphism classes of simplicial complexes and are finite, then $d(K_1,K_2) \in \Q$ (cor. \ref{cor:findKisoQ}). 
\item If $K_1$, $K_2$ are isomorphism classes of finite dimensional complexes, then they admit representatives $\KK_1,\KK_2$ such that $d(K_1,K_2) = d(\KK_1,\KK_2)$ (prop. \ref{prop:findimdd}).
\item If the $(n+1)$-squeletons of $K_1$ and $K_2$ differ, then $d(K_1,K_2) \geq \frac{1}{n+2}$ (cor. \ref{cor:ineqsquela})
\item The distance function takes all possible values inside $[0,1]$, except possibly
irrational values in $[1/2,1]$ (prop. \ref{prop:excompiso}).
\end{itemize}

Finally, we explicitely compute the distances between isomorphism classes of
simplicial complexes with at most 4 vertices.

\medskip

{\bf Acknowledgements.} I thank Craig Westerland for an inspiring discussion at a conference in Edinburgh which lead to this project.

\section{Hausdorff metric on simplicial complexes}
\subsection{Definition}

Let $\Omega$ a standard measured space, that is a measured space
isomorphic to $[0,1]$. Let $S$ denote a fixed set, and $\mathcal{P}(S)$ (resp. $\mathcal{P}^*(S)$) the collection
of all subsets (resp. all non-empty subsets) of $S$. If 
$\mathcal{K} \subset \mathcal{P}(S)$ is non-empty,
we define $L(\mathcal{K}) = \{  f : \Omega \to S \ | \ f(\Omega) \in \mathcal{K} \}$. This is a subspace, introduced in \cite{SRV}, of the space $L(\Omega,S)$ of
Borel maps $\Omega \to S$ up to neglectability, as defined in \cite{CCS,SRV}. It is a metric space, where $d(f,g) = \int_{\Omega} d(f(t),g(t))\dd t$, where $S$ is endowed with the discrete metric. 
Then, $L(\mathcal{K})$ is closed
iff $\mathcal{K}$ is finite. We denote its closure by $\bar{L}(\mathcal{K})$. We denote $\PF^*(S) \subset \mathcal{P}^*(S)$ the collection of all finite, non-empty subsets.

If $\mathcal{K}_1,\mathcal{K}_2 \subset \mathcal{P}_f^*(S)$, $L(\mathcal{K}_1)= L(\mathcal{K}_2)$ if and only if $\bar{L}(\mathcal{K}_1)= \bar{L}(\mathcal{K}_2)$. This is because
$L(\mathcal{K}) = \{ f \in \bar{L}(\Omega,\mathcal{K}) ; \# f(\Omega) < \infty \}$. In \cite{SRV}, it is proven that, if $\mathcal{K}$ is a simplicial complex, then $L(\KK)$ (resp. $\bar{L}(\KK)$) is
homotopically equivalent (resp. weakly homotopically equivalent) to the usual \emph{geometric realization} of $\mathcal{K}$.

For $\mathcal{K}_1,\mathcal{K}_2 \subset \mathcal{P}_f^*(S)$,
we define the \emph{Hausdorff distance} $d(\mathcal{K}_1,\mathcal{K}_2)$ as the usual Hausdorff
distance $d(L(\mathcal{K}_1),L(\mathcal{K}_2))$ between the subspaces $L(\mathcal{K}_1)$
and  $L(\mathcal{K}_2)$ namely $d(X,Y) = \max( \delta(X,Y),\delta(Y,X))$ with
$$
\delta(X,Y) = \sup_{x \in X} d(x,Y) = \sup_{x \in X} \inf_{y \in Y} d(x,y).
$$
Recall that the Hausdorff distance defines a pseudo-distance on the collection of bounded subspaces of any given
metric space $E$, that $\delta(\bar{X},\bar{Y}) = \delta(X,Y)$, $d(\bar{X},\bar{Y}) = d(X,Y)$,
and that it restricts to a distance on the collection of closed bounded supspaces.

We then have $d(\mathcal{K}_1,\mathcal{K}_2) = \max(\delta(\mathcal{K}_1,\mathcal{K}_2),\delta(\mathcal{K}_2,\mathcal{K}_1))$
with
$$
\delta(\mathcal{K}_1,\mathcal{K}_2) = \delta\left(L(\mathcal{K}_1),L(\mathcal{K}_2)\right) = \sup_{f \in L(\mathcal{K}_1)} d(f,L(\mathcal{K}_2)) 
= \sup_{f \in L(\mathcal{K}_1)} \inf_{g \in L(\mathcal{K}_2)} d(f,g).
$$
We notice that $\delta(\mathcal{K}_1,\mathcal{K}_2) = \delta\left(L(\mathcal{K}_1),L(\mathcal{K}_2)\right)
= \delta\left(\bar{L}(\mathcal{K}_1),\bar{L}(\mathcal{K}_2)\right)$, and
$d(\mathcal{K}_1,\mathcal{K}_2) = d\left(L(\mathcal{K}_1),L(\mathcal{K}_2)\right)
= d\left(\bar{L}(\mathcal{K}_1),\bar{L}(\mathcal{K}_2)\right)$. Therefore
$$
\mathcal{K}_1 = \mathcal{K}_2 \Leftrightarrow L(\mathcal{K}_1) = L(\mathcal{K}_2)
\Leftrightarrow \bar{L}(\mathcal{K}_1) = \bar{L}(\mathcal{K}_2)
\Leftrightarrow d(\bar{L}(\mathcal{K}_1), \bar{L}(\mathcal{K}_2)) = 0
\Leftrightarrow d(L(\mathcal{K}_1), L(\mathcal{K}_2)) = 0.
$$

We will see at the end of section \ref{sect:dLK} that, if $S \subset T$, and $\mathcal{K}_1,\mathcal{K}_2$ are simplicial complexes over $S$, hence over $T$, then the distance
between the two does not depend on whether it is calculated over $S$ or over $T$.
Therefore it is legitimate not to include $S$ in the notation of
$d(\mathcal{K}_1,\mathcal{K}_2)$.

We denote $\mathcal{T} = \mathcal{P}^*(\mathcal{P}^*_f)$ endowed with this distance,
and $\mathcal{S} \subset \mathcal{T}$ its subspace made of the simplicial complexes,
namely the elements $E$ of $\mathcal{T}$ satisfying $\forall F_1,F_2 \in \mathcal{P}^*_f(S) \  F_1 \subset F_2 \ \& \ F_2 \in E \Rightarrow F_1 \in E$.

\subsection{Computation means, estimates, and rationality}
\subsubsection{Distance to $L(\mathcal{K})$}
\label{sect:dLK}

The distance of a given map to $L(\mathcal{K})$ can be easily determined. It is given by the following lemma.

\begin{lemma} \label{lem:calcdFK} Let $f_0 \in L(\Omega,S)$ and $\mathcal{K} \subset \mathcal{P}_*(S)$ a simplicial complex. Then
$$
d(f_0,L(\mathcal{K})) = \inf_{F \in \mathcal{E}} \la\left( \Omega \setminus f_0^{-1}(F)\right)
$$
where $\mathcal{E}$ is the set of all maximal subsets $F$ of $f_0(\Omega)$ having the property $F \in \mathcal{K}$, and by
convention this infimum is $1 = diam(L(\Omega,S))$ if $\mathcal{E}$ is empty. Moreover, if $d(f_0,L(\mathcal{K}))<1$ then there exists $f \in L(\mathcal{K})$ with $f(\Omega) \subset
f_0(\Omega)$ such that $d(f_0,f) = d(f_0,L(\mathcal{K}))$.
\end{lemma}
\begin{proof} By definition we have that
$d(f_0,L(\mathcal{K}))$ is the infimum of the
$d(f_0,f)$ for $f$ in $L(\mathcal{K})$. We first show that one only
needs to consider the $f$ satisfying the following properties :
\begin{enumerate}
\item $f(\Omega) \subset f_0(\Omega)$. For, if $f \in L(\mathcal{K})$
with $f(\Omega) \not\subset f_0(\Omega)$, let $F \subset f_0(\Omega) \cap f(\Omega)$
maximal for this property. If $f_0(\Omega) \cap f(\Omega) = \emptyset$
then $d(f_0,f) = 1$ and if this is the case for all $f$ the statement is true by our convention
on the infimum. So we can assume $F \neq \emptyset$. Let $x_0 \in F$. We define $\tilde{f} : \Omega \to S$
by $t\mapsto f(t)$ for $t \in f^{-1}(F)$ and $t \mapsto x_0$ on $\Omega \setminus f^{-1}(F)$. Then
$$
d(f_0,\tilde{f}) = \int_{f^{-1}(F)} d(f_0(t),\tilde{f}(t)) \dd t + \int_{ ^c f^{-1}(F)} d(f_0(t),x_0)\dd t
$$
But $\{ t \in\, ^c f^{-1}(F) \ | \ f_0(t) = f(t) \}$ has measure $0$, for otherwise, since $f_0(\Omega)$ is countable
there would be $\Omega_0 \subset f^{-1}(F)$ of positive measure and $x_1 \in S$ such that $f_0(\Omega_0) = f(\Omega_0) = \{ x_1 \}$.
But then $x_1 \not\in F$ since $\Omega_0 \subset \, ^c f^{-1}(F)$, hence $G = F \sqcup \{ x_1 \}$ is a subset of $f_0(\Omega_0) \cap
f(\Omega)$ larger than $F$ hence contradicting its maximality.

Therefore $\int_{ ^c f^{-1}(F)} d(f_0(t),x_0)\dd t \leq \int_{ ^c f^{-1}(F)} d(f_0(t),f(t))\dd t$ and, since
$f(t) = \tilde{f}(t)$ for all $t \in f^{-1}(F)$, we get $d(f_0,\tilde{f}) \leq d(f_0,f)$.
\item $f(\Omega) \subset f_0(\Omega)$ with $f(\Omega)$ maximal among the subsets of $f_0(\Omega)$ belonging
to $\mathcal{K}$. For, if $f \in L(\mathcal{K})$ with $f(\Omega) = F \subset f_0(\Omega)$ and $F \subsetneq G \subset f_0(\Omega)$
with $G \in \mathcal{K}$, then, defining $\tilde{f}$ by $\tilde{f}(t) = f_0(t)$ for $t \in f_0^{-1}(G )$ and
$\tilde{f}(t) = f(t)$ otherwise, we have $\tilde{f}(\Omega) \subset G \in \mathcal{K}$ hence $\tilde{f} \in L(\mathcal{K})$,
and $d(\tilde{f},f_0)< d(f,f_0)$.
\item $f(\Omega) \subset f_0(\Omega)$ with $f(\Omega)$ maximal among the subsets of $f_0(\Omega)$ belonging
to $\mathcal{K}$ with $f(t) = f_0(t)$ for all $t \in f_0^{-1}(f(\Omega))$. For, if $f$ satisfies the previous conditions
and we let $\tilde{f}$ being defined by $\tilde{f}(t) = f_0(t) $ for all $t \in f_0^{-1}(f(\Omega))$, we have $\tilde{f}(\Omega)
= f(\Omega)$ and $d(f_0,\tilde{f}) \leq d(f_0,f)$.
\end{enumerate}
Let then $f : \Omega \to S$ with $F = f(\Omega)$ maximal among the subsets of $f_0(\Omega)$ belonging
to $\mathcal{K}$ with $f(t) = f_0(t)$ for all $t \in f_0^{-1}(f(\Omega))$. Then
$$
d(f_0,f) = \la\left(\Omega \setminus f_0^{-1}(f(\Omega)) \right).
$$
Since all subsets of $f_0(\Omega)$ belonging to $\mathcal{K}$ can be realized as $f(\Omega)$ for some $f \in L(\Omega,S)$,
this proves the claim.
\end{proof}

From this property, one sees easily why the distance does not depend on the
chosen vertex set $S$. Indeed, if $S \subset T$ and $d_S,d_T$ are the
two avatars of the distance, with the two auxiliary functions $\delta_S,\delta_T$,
for any simplicial complexes $\mathcal{K}_1,\mathcal{K}_2$ over $S$, we get that
$$
d_S(f_0,L(\KK_2)) = \inf_{F \in \mathcal{E}} \la\left( \Omega \setminus f_0^{-1}(F)\right) = d_T(f_0,L(\KK_2))
$$
hence 
$$
\delta_T(\KK_1,\KK_2) = \sup_{\stackrel{f_0 \in L(\Omega,T)}{f_0(\Omega) \in \KK_1}} d_T(f_0,L(\KK_2))
= \sup_{\stackrel{f_0 \in L(\Omega,S)}{f_0(\Omega) \in \KK_1}} d_T(f_0,L(\KK_2))
= \sup_{\stackrel{f_0 \in L(\Omega,S)}{f_0(\Omega) \in \KK_1}} d_S(f_0,L(\KK_2))
\delta_S(\KK_1,\KK_2)
$$
and $d_T(\KK_1,\KK_2) = d_S(\KK_1,\KK_2)$.

\medskip
 
For $F \in \mathcal{P}_f^*(S)$, one defines $d(F,\mathcal{K}) = \sup_{f_0(\Omega) = F} d(f_0,L(\mathcal{K}))$.

We denote $\mathcal{T}_n = \{ E \in \mathcal{T} \ | \ \forall F \in E \ \# F \leq n+1 \}$ and $\mathcal{S}_n = \mathcal{T}_n \cap \mathcal{S}$. Let us choose
 $F \in \mathcal{P}_f^*(S)$ with $|F| \leq n+1$ and $\KK \in \mathcal{S}$. Then
$$d(F,\mathcal{K}) = \sup_{f_0(\Omega) = F} d(f_0,L(\mathcal{K})) =
\sup_{f_0(\Omega) = F} \min_{G \in \mathcal{E}} \la\left( \Omega \setminus f_0^{-1}(G)\right) 
$$
with $\mathcal{E}$ is the set of all maximal subsets $G$ of $F$ having the property $G \in \mathcal{K}$. Now, such a number is equal to some
$$
\sup_{f(\Omega) = \{ 0,1\dots, n \}} \min_{G \in \mathcal{C}} \la\left( \Omega \setminus f_0^{-1}(G)\right)  
$$ for $\mathcal{C}$ a collection of sets inside $\PF^*(\{0,\dots, n \})$. Since there is only a finite number of such collections, the set of all possible
values of $d(F,\KK)$ for $|F|\leq n+1$ and $\KK$ running among all possible simplicial complexes is finite. An immediate consequence is the following.

\begin{proposition} \label{prop:finSn}
For any given $n,m \in \N$, the map $d : \mathcal{S}_n \times \mathcal{S}_m \to [0,1]$ has finite image.
\end{proposition}

\subsection{Contracting property of the intersection}

We prove the following.

\begin{proposition} \label{prop:intersectcontract} Let $\mathcal{A},\mathcal{B},\mathcal{K}$ be simplicial complexes over $S$. Then
$\mathcal{A}\cap \mathcal{K}$ and $\mathcal{B}\cap \mathcal{K}$ are simplicial complexes and
$$
d(\mathcal{A} \cap \mathcal{K}, \mathcal{B}\cap \mathcal{K})
\leqslant d(\mathcal{A},\mathcal{B})
$$
\end{proposition}
\begin{proof}
By symmetry, one only needs to prove
$\delta(\mathcal{A} \cap \mathcal{K}, \mathcal{B}\cap \mathcal{K})
\leqslant \delta(\mathcal{A},\mathcal{B})$.
We have $\delta(\mathcal{A},\mathcal{B}) = \sup_{F \in \mathcal{A}} d(F,\mathcal{B})$
and 
$$
\delta(\mathcal{A} \cap \mathcal{K},\mathcal{B}\cap \mathcal{K})= \sup_{F \in \mathcal{A}\cap \mathcal{K}} d(F,\mathcal{B}\cap \mathcal{K})
\leqslant \sup_{F \in \mathcal{A}} d(F,\mathcal{B}\cap \mathcal{K}).
$$
Therefore, it is sufficient to prove 
$d(F,L(\mathcal{B})= d(F,L(\mathcal{B}\cap \mathcal{K})$
for any $F \in \mathcal{A} \cap \mathcal{K}$.
Now, by lemma \ref{lem:calcdFK} we have
$$
d(F, \mathcal{B}\cap \mathcal{K}) =
\sup_{f_0(\Omega) = F} d(f_0,\mathcal{B}\cap \mathcal{K})
= \sup_{f_0(\Omega) = F} \inf_{G \in \mathcal{E}}
\la\left( \Omega \setminus f_0^{-1}(G)\right)
$$
where $\mathcal{E}$ is the set of maximal elements
of $\mathcal{F} = \{ \emptyset \neq G \subset F ; G \in \mathcal{B} \cap \mathcal{K} \}$. But since
$F \in \mathcal{K}$, we have $\mathcal{F} = \{ \emptyset \neq G \subset F ; G \in \mathcal{B}  \}$
whence $d(F, \mathcal{B}\cap \mathcal{K}) =  d(F, \mathcal{B})$ and this proves the claim.

\end{proof}

As a consequence, we get the following. Recall that we denote $\mathcal{T}_n = \{ E \in \mathcal{T} \ | \ \forall F \in E \ \# F \leq n+1 \}$ and $\mathcal{S}_n = \mathcal{T}_n \cap \mathcal{S}$.
We have a natural injection $j_n : \mathcal{T}_n \into \mathcal{T}$ as well as a natural
projection $\pi_n : \mathcal{T} \onto \mathcal{T}_n$, defined by $\pi_n(E) = \{ F \in E \ | \ \# F \leq n+1 \}$. One easily checks that
they restrict to maps $\mathcal{S}_n \into \mathcal{S}$ and $\mathcal{S} \onto \mathcal{S}_n$. We have $\pi_n \circ j_n = \mathrm{Id}$,
and $\pi_n \circ \pi_m = \pi_{\min(m,n)}$. We have

\begin{proposition} \label{prop:pinlip} The map $\pi_n : \mathcal{S} \onto \mathcal{S}_n$ is $1$-Lipschitz. Moreover, for
every $\mathcal{K}_1,\mathcal{K}_2 \in \mathcal{S}$, we have
$$
d(\mathcal{K}_1,\mathcal{K}_2) = \sup_n d(\pi_n(\mathcal{K}_1),\pi_n(\mathcal{K}_2)) = \sup_{X \mbox{ finite}} d(\mathcal{K}_1 \cap \PF^*(X),\mathcal{K}_2 \cap \PF^*(X))
$$
\end{proposition}
\begin{proof}
Let $\Delta(n) = \{ F \in \mathcal{P}_f^*(S) ; |F| \leq n+1 \}$. Then $\Delta(n)$ is a simplicial complex,
and, for any $\mathcal{K}$, we have $\pi_n(\mathcal{K})
 = \mathcal{K} \cap \Delta(n)$. From the above proposition this implies that each $\pi_n$ is 1-Lipschitz.

Now, $\delta(\mathcal{K}_1,\mathcal{K}_2)$ is equal to
$$
\delta(\mathcal{K}_1,\mathcal{K}_2)=
 \sup_{f \in L(\mathcal{K}_1)} d(f,L(\mathcal{K}_2))
= \sup_{n \geq 0} \sup_{\stackrel{f \in L(\mathcal{K}_1)}{\{ |f(\Omega)| \leq n+1\}}} d(f,L(\mathcal{K}_2))
= \sup_{n \geq 0} \sup_{\stackrel{f \in L(\mathcal{K}_1)}{\{ |f(\Omega)| \leq n+1\}}} d(f,L(\pi_n(\mathcal{K}_2)))
$$
{}
$$
= \sup_{n \geq 0} \sup_{f \in L(\pi_n(\mathcal{K}_1))} d(f,\pi_n(L(\mathcal{K}_2))) = \sup_{n\geq 0}\delta(\pi_n(\mathcal{K}_1),\pi_n(\mathcal{K}_2))
$$
hence $d(\mathcal{K}_1,\mathcal{K}_2) = \sup_n d(\pi_n(\mathcal{K}_1),\pi_n(\mathcal{K}_2))$. Finally,
for any $X$ we have $d(\mathcal{K}_1 \cap \PF^*(X),\mathcal{K}_2 \cap \PF^*(X))\leq d(\mathcal{K}_1,\mathcal{K}_2)$ hence
$\sup_{X \mbox{ finite}} d(\mathcal{K}_1 \cap \PF^*(X),\mathcal{K}_2 \cap \PF^*(X)) \leq d(\mathcal{K}_1,\mathcal{K}_2)$. Now,
$$\delta(\mathcal{K}_1,\mathcal{K}_2)
= \sup_{f \in L(\mathcal{K}_1)} d(f,L(\mathcal{K}_2))
=\sup_{X \mbox{ finite}} \sup_{f \in L(\mathcal{K}_1\cap \PF^*(X))} d(f,L(\mathcal{K}_2))
$${}$$\leqslant
 \sup_{X \mbox{ finite}} \sup_{f \in L(\mathcal{K}_1\cap \PF^*(X))} d(f,L(\mathcal{K}_2\cap \PF^*(X)))
 \leqslant
 \sup_{X \mbox{ finite}} \delta(\mathcal{K}_1\cap \PF^*(X),\mathcal{K}_2\cap \PF^*(X))
 $${}$$
 \leqslant
 \sup_{X \mbox{ finite}} d(\mathcal{K}_1\cap \PF^*(X),\mathcal{K}_2\cap \PF^*(X)) 
 $$
 hence
$$
d(\mathcal{K}_1,\mathcal{K}_2) =
\max(\delta(\mathcal{K}_1,\mathcal{K}_2),\delta(\mathcal{K}_2,\mathcal{K}_1)) \leqslant
 \sup_{X \mbox{ finite}} d(\mathcal{K}_1\cap \PF^*(X),\mathcal{K}_2\cap \PF^*(X)) 
\leqslant d(\mathcal{K}_1,\mathcal{K}_2)
$$ 
and this concludes the proof.
\end{proof}

\begin{corollary} \label{cor:dist1vertex}
For every $\mathcal{K}_1,\mathcal{K}_2 \in \mathcal{S}$,
$d(\mathcal{K}_1,\mathcal{K}_2) = 1$ unless
$\bigcup \mathcal{K}_1 = \bigcup \mathcal{K}_2$.
\end{corollary}

\begin{proof}
We have $d(\mathcal{K}_1,\mathcal{K}_2) \geq
d(\pi_0(\mathcal{K}_1),\pi_0(\mathcal{K}_2))$. Now, if $\mathcal{K}_1,\mathcal{K}_2$ have only 0-simplices, it is easily checked that $d(\mathcal{K}_1,\mathcal{K}_2) \in \{ 0, 1 \}$ with  $d(\mathcal{K}_1,\mathcal{K}_2) = 1$ if and only if $\mathcal{K}_1=\mathcal{K}_2$, which is clearly equivalent to $\bigcup\mathcal{K}_1 = \bigcup\mathcal{K}_2$.
\end{proof}

We now consider the probability law map $\Psi : L(\Omega,S) \to M(S) = |\mathcal{P}_f^*(S)|_1$ and the induced
map $L(\KK) \to |\KK|_1$ which were studied in \cite{SRV}. Every $\overline{|\KK|_1}$ being closed and
bounded inside $|\mathcal{P}_f^*(S)|_1$ we can consider the Hausdorff distance between two of them,
with respect to the original metric on
$M(S)$
given by 
$$
d(\alpha,\beta) =
\frac{1}{2}|\alpha-\beta|_1 = \frac{1}{2}\sum_{s \in S} |\alpha(s)- \beta(s)|$$
We denote $S_{\alpha}$ the support of $\alpha$.
We notice that, by lemma \ref{lem:calcdFK}, $d(f_0,L(\KK))$
depends only on $\Psi(f_0)$ and $\KK$.

\begin{proposition} For $\alpha : S \to [0,1]$ with $\sum_s \alpha(s)=1$, and $\mathcal{K}$ a simplicial complex,
we have
$$
d(\alpha,|\KK|_1) = \frac{1}{2}\inf_{\stackrel{S_{\beta} \in \KK}{S_{\beta} \subset S_{\alpha}}} \sum_{s \in S_{\alpha}}|\alpha(s) - \beta(s)| =  \inf_{\stackrel{F \in \KK}{F \subset S_{\alpha}}} \left( 1 - \sum_{s \in F} \alpha(s) \right)
$$
When $f_0 \in L(\Omega,S)$, we have
$$
d(\Psi(f_0) , |\KK|_1) =  d(f_0,L(\KK)).
$$
\end{proposition}
\begin{proof}
Let $\beta \in |\KK|_1$, and define $\hat{\beta} : S \to [0,1]$ by $\hat{\beta}(s) = 0$ if $s \not\in S_{\alpha}$,
and $\hat{\beta}(s) = \beta(s) + K$ is $s \in S_{\alpha}$
with 
$$
K = \frac{\sum_{s \not\in S_{\alpha}} \beta(s)}{\#S_{\alpha}\cap S_{\beta}}
$$ 
Then 
$$
|\alpha - \hat{\beta}|_1 = \sum_{s \in S_{\alpha} \setminus S_{\beta}} \alpha(s) + \sum_{s \in S_{\alpha} \cap S_{\beta}} |\alpha(s) - \beta(s) - K|
\leqslant  \sum_{s \in S_{\alpha} \setminus S_{\beta}} \alpha(s) + \sum_{s \in S_{\alpha} \cap S_{\beta}} |\alpha(s) - \beta(s)  | + K \# S_{\alpha} \cap S_{\beta}
$$
{}
$$
\leqslant  \sum_{s \in S_{\alpha} \setminus S_{\beta}} \alpha(s) + \sum_{s \in S_{\alpha} \cap S_{\beta}} |\alpha(s) - \beta(s)  | + \sum_{s \not\in S_{\alpha}} \beta(s) = |\alpha - \beta|_1
$$
and it follows that 
$$
d(\alpha,|\KK|_1) = \frac{1}{2}\inf_{\stackrel{S_{\beta} \in \KK}{S_{\beta} \subset S_{\alpha}}} |\alpha - \beta|_1
= \frac{1}{2}\inf_{\stackrel{S_{\beta} \in \KK}{S_{\beta} \subset S_{\alpha}}} \sum_{s \in S_{\alpha}}|\alpha(s) - \beta(s)|
$$
Now let $F \in \KK$ with $F \subset S_{\alpha}$ and $\beta \in L(\mathcal{K})$ with $S_{\beta} \subset F$. Assume that
$\beta(s_0) < \alpha(s_0)$ for some $s_0 \in F$. Since
$$
\sum_{s \in F} \beta(s) - \alpha(s) = 1 - \sum_{s \in F}\alpha(s) \geqslant 0
$$
there exists $s_1 \in F$ with $\beta(s_1) > \alpha(s_1)$. Let us set $\hat{\beta} : S \to [0,1]$ defined by
$$
\begin{array}{lcll}
s & \mapsto & \beta(s) & \mbox{ if } s \not\in \{ s_0,s_1 \} \\
s_0 & \mapsto & \beta(s_0)+\eps \\
s_1 & \mapsto & \beta(s_1)-\eps 
\end{array}
$$
Then $S_{\hat{\beta}} \subset F$ and 
$$
\left| \hat{\beta}(s_1) - \alpha(s_1) \right| = 
\hat{\beta}(s_1) - \alpha(s_1) = \beta(s_1) - \alpha(s_1) - \eps 
= \left| \beta(s_1) - \alpha(s_1) \right| - \eps  
$$
while
$$
\left| \hat{\beta}(s_0) - \alpha(s_0) \right| = 
\alpha(s_0)-\hat{\beta}(s_0)   = \alpha(s_0) - \beta(s_0)  - \eps 
= \left|  \alpha(s_0) - \beta(s_0) \right| - \eps  
$$
Now, we have
$$
\sum_{s \in S_{\alpha} } \left|\alpha(s) - \beta(s)\right| =
\sum_{s \in S_{\alpha} \setminus S_{\beta}} \alpha(s) +
Z(\beta) \mbox{ with } Z(\beta) =  \sum_F 
\left| \beta(s) - \alpha(s) \right|
$$
and
$Z(\hat{\beta}) \leq Z(\beta) - 2 \eps$. It follows that,
when computing the infimum, we can assume
that $\beta(s) \geq \alpha(s)$ for all $s \in F$, whence
$$
\inf_{\stackrel{S_{\beta} \in \KK}{S_{\beta} \subset S_{\alpha}}} \sum_{s \in S_{\alpha}}|\alpha(s) - \beta(s)|
= \inf_{\stackrel{S_{\beta} \in \KK}{S_{\beta} \subset S_{\alpha}}} 
\sum_{s \in S_{\alpha} \setminus S_{\beta}} \alpha(s) +
\sum_F 
 \left(\beta(s) - \alpha(s) \right)
= \inf_{\stackrel{S_{\beta} \in \KK}{S_{\beta} \subset S_{\alpha}}} 
\sum_{s \in F} \beta(s) + \sum_{s \in S_{\alpha}} \alpha(s) - 2 \sum_{s \in F}\alpha(s)
$$
that is
$$
2\left( 1  - \sum_{s \in F}\alpha(s) \right)
$$
and this proves the claim, the last equality being then
an obvious consequence of lemma \ref{lem:calcdFK}.

\end{proof}

\begin{corollary} \label{cor:dLKdMK}
Let $\mathcal{K}$ and $\mathcal{K}'$ denote two simplicial
complexes over $S$. Let us also denote $d$ the Hausdorff pseudo-distance between bounded subsets of $M(S)$. Then
$$
d\left(|\mathcal{K}|_1,|\mathcal{K}'|_1\right) =  d\left(L(\mathcal{K}), L(\mathcal{K}')\right).
$$ 
\end{corollary}
\begin{proof}
We have
$$
\delta(|\mathcal{K}|_1,|\mathcal{K}'|_1) = \sup_{\mu \in |\mathcal{K}_1|} d(\mu, |\mathcal{K}'_1|)
 = \sup_{f_0 \in L(\mathcal{K})} d(\Psi(f_0), |\mathcal{K}'_1|)
 $$
by surjectivity of $\Psi$. Now, by the proposition, this is equal to
$$
 \sup_{f_0 \in L(\mathcal{K})}  d(f_0, L(\mathcal{K}'))
= 
 \delta(L(\mathcal{K}), L(\mathcal{K}'))$$
whence $d(|\mathcal{K}|_1,|\mathcal{K}'|_1)$ is equal to
$$
 \max\left(\delta(|\mathcal{K}|_1,|\mathcal{K}'|_1),
\delta(|\mathcal{K}'|_1,|\mathcal{K}|_1)\right)
= \max\left(\delta(L(\mathcal{K}),L(\mathcal{K}')),
\delta(L(\mathcal{K}'),L(\mathcal{K}))\right)
= d(L(\mathcal{K}),L(\mathcal{K}')
$$
and this proves the claim.
\end{proof}
\subsubsection{Upper and lower bounds}

We first need a technical lemma.

\begin{lemma} \label{lem:optimrparts} Let  $n \geq 1$, $1 \leq r \leq n$, $E = \{1,\dots, n \}$, and $\Delta_n = \{ (a_1,\dots,a_n) \in [0,1]^n \ | \ \sum_{i=1}^n a_i = 1 \}$ the $n$-dimensional
simplex. Then
$$
\sup_{ (a_i)_{i=1,\dots,n} \in \Delta_n } \min_{\stackrel{D \subset E}{|D|=r}} \left( 1 - \sum_{i \in D} a_i \right) = \frac{n-r}{n}
$$
\end{lemma}
\begin{proof} Let us denote $\beta(n,r)$ the LHS. For $(a_1,\dots,a_n) = (1/n,\dots,1/n) \in \Delta_n$,
we have
$$
\min_{\stackrel{D \subset E}{|D|=r}} \left( 1 - \sum_{i \in D} a_i \right)
= \min_{\stackrel{D \subset E}{|D|=r}} \frac{n-r}{n} = \frac{n-r}{n}
$$
hence $\beta(n,r) \geq (n-r)/r$. If we had $\beta(n,r) > (n-r)/r$ then there would exist $(a_1,\dots,a_n) \in \Delta_n$
such that, for all  $D \subset E$ with $|D| = r$, we have $1 - \sum_{i \in D} a_i > (n-r)/n$
that is $\sum_{i \in D} a_i < r/n$. But this implies
$$
\sum_{\stackrel{D \subset E}{|D| = r}} \sum_{i \in D} a_i < \sum_{\stackrel{D \subset E}{|D| = r}} \frac{r}{n} = 
\left( \begin{array}{c} n \\ r \end{array} \right) \frac{r}{n} = \left( \begin{array}{c} n-1 \\ r-1 \end{array} \right) 
$$
and on the other hand we have
$$
\sum_{\stackrel{D \subset E}{|D| = r}} \sum_{i \in D} a_i
 = \sum_{i=1}^n \sum_{\begin{array}{c} D \subset E \\ |D| = r \\ i \in D \end{array}} a_i
 = \sum_{i=1}^n a_i \times  \# \{ D \subset E \setminus \{ i \} ;   |D| = r-1 \}
 $$
 hence
$$
\sum_{\stackrel{D \subset E}{|D| = r}} \sum_{i \in D} a_i 
 = \left( \sum_{i=1}^n a_i \right) \times \left( \begin{array}{c} n-1 \\ r-1 \end{array} \right) =\left( \begin{array}{c} n-1 \\ r-1 \end{array} \right) 
$$
and this contradiction proves the claim.
\end{proof}

Recall that, for $F \in \mathcal{P}_f^*(S)$, we use the notation $d(F,\mathcal{K}) = \sup_{f_0(\Omega) = F} d(f_0,L(\mathcal{K}))$. 

\begin{proposition} \label{prop:uplowbound} Let $F \in \mathcal{P}_f^*(S)$ with $|F| = n \geq 1$, and $\mathcal{K} \subset \mathcal{P}_{f}^*(S)$.
\begin{enumerate}
\item If $H$ is a maximal subset of $F$ belonging to $\mathcal{K}$, then $d(F,\mathcal{K}) \geq 1 - |H|/n$. In particular, if $F \not\in \mathcal{K}$,
then $d(F, \mathcal{K}) \geq 1/n$.
\item If $\mathcal{K}$ contains all subsets of cardinality $r$ of $F$ for some $r \geq 1$, then $d(F, \mathcal{K}) \leq 1 - r/n$.
\end{enumerate}
\end{proposition}
\begin{proof}
By lemma \ref{lem:calcdFK} this is equal to
$$
d(F,\mathcal{K}) = \sup_{f_0(\Omega) =F } \min_{\stackrel{H < F}{H \in \mathcal{K}}} \la \left( \Omega \setminus f_0^{-1}(H) \right)
$$
and the $H$ in the minimum can be taken to be maximal. In order to get an approximation of this quantity, we first need a lemma.

If $F \in \mathcal{K}$, then for all $f_0 : \Omega \to S$, $f_0(\Omega) = F$ implies $d(f_0,L(\mathcal{K})) = 0$,
hence $d(F,\mathcal{K}) = 0$. If not, let $n = |F|$. For $H \subset F$ and $H \in \mathcal{K}$
maximal for these properties, let $r_H = |H|$. We choose $f_{00} : \Omega \to F$ with uniform law. Then,
for all $H$ as above, we have $\la(\Omega \setminus f_{00}^{-1}(H)) = (n-r_H)/n$
hence $d(F,\mathcal{K}) \geq 1- r_H/n$ for all such $H$. This inequality is obviously valid in the case $F \in \mathcal{K}$.

Let us now consider an arbitrary $f_0 : \Omega \to S$ with $f_0(\Omega) = F$, and let $\mathcal{E}$
the set of all 
$H \subset F$ and $H \in \mathcal{K}$ which are
maximal for these properties.  Assume that all the subsets of $F$ of cardinality $r$ belong to $\mathcal{K}$.
Then
we have
$$
\forall H \in \mathcal{E} \ \forall D \subset H \ |D| = r \Rightarrow \la \left( \Omega \setminus f_0^{-1}(H) \right)
\leqslant \la\left( \Omega \setminus f_0^{-1}(D) \right).
$$
Therefore, by considering the law $(a_x)_{x \in F}$ on $F$ associated to $f_0$ (that is $a_x = \la(f_0^{-1}(\{ x \}))$) we get
$$
d(F,L_2) \leqslant \sup_{(a_x)_{x \in F} \in \R_+^F ; \sum_x a_x = 1 }  \min_{\stackrel{D \subset F}{|D|=r}} \left( 1 - \sum_{x \in D}
a_x \right) = \frac{n-r}{r}
$$
by lemma \ref{lem:optimrparts}, and this concludes the proof.
\end{proof}

\begin{corollary} \label{cor:ineqsquel} Let $\mathcal{K}_1, \mathcal{K}_2 \in \mathcal{S}$. If $\mathcal{K}_1\neq \mathcal{K}_2$ then there exists $N = N(\mathcal{K}_1,\mathcal{K}_2)$ such that $\pi_N(\mathcal{K}_1)=\pi_N( \mathcal{K}_2)$ and
$\pi_{N+1}(\mathcal{K}_1)\neq \pi_{N+1}( \mathcal{K}_2)$. We then have
$$
\frac{1}{N+2} \leqslant
d(\mathcal{K}_1, \mathcal{K}_2)
$$
\end{corollary}
\begin{proof}
By definition $\pi_N(\mathcal{K}_1)=\pi_N( \mathcal{K}_2)$ means that, for all $F \in \mathcal{P}_f^*(S)$
with $|F| \leq N+1$, then $F \in \mathcal{K}_1 \Leftrightarrow F \in \mathcal{K}_2$. 
Now, we know that $d(\pi_{N+1}(\mathcal{K}_1) ,\pi_{N+1}(\mathcal{K}_2)) \leq d(\mathcal{K}_1,\mathcal{K}_2)$
by proposition \ref{prop:pinlip}. In order to prove the lower bound we can thus assume $\mathcal{K}_1,\mathcal{K}_2 \in \mathcal{S}_{N+1}$. Since
$\pi_N(\mathcal{K}_1) = \pi_N(\mathcal{K}_2)$ we have
$$
\delta(\mathcal{K}_1,\mathcal{K}_2) = \sup_{F  \in \mathcal{K}_1} d(F,\mathcal{K}_2)
= \sup_{\stackrel{F \in \mathcal{K}_1}{|F| = N+2}} d(F,\mathcal{K}_2).
$$
Therefore, by the proposition, if there exists $F \in \mathcal{K}_1 \setminus \mathcal{K}_2$ we have $\delta(\mathcal{K}_1,\mathcal{K}_2) \geq 1/(N+2)$,
and otherwise $\delta(\mathcal{K}_1,\mathcal{K}_2)= 0$. It follows that $\mathcal{K}_1 \neq \mathcal{K}_2$
implies $d(\mathcal{K}_1,\mathcal{K}_2) \geq 1/(N+2)$ and this proves the claim.
\end{proof}

\subsubsection{Computation and rationality for finite complexes}

We explain how this metric between two given finite complexes $\KK_1$ and $\KK_2$ can be explicitely computed. As a byproduct, the algorithm we provide will prove the following.

\begin{proposition} \label{prop:dKKQ}
If $\KK_1$ and $\KK_2$ are finite complexes, then $d(\KK_1,\KK_2) \in \Q$.
\end{proposition}

In order to compute $d(\KK_1,\KK_2)$ we need to explain how to compute $\delta(\KK_1,\KK_2)$. Since $\KK_1$ has a finite number of simplices, this amounts to computing a
finite number of terms of the form $d(F,\KK_2)$. Therefore we only need to compute and to prove the rationality of
$$
\sup_{\stackrel{(x_s)_{s \in F} \in [0,1]^F}{\sum x_s = 1}} \inf_{\stackrel{G \in \KK}{G \subset F}} \left( 1 - \sum_{s \in G} x_s\right) = 1 - D(F,\KK) \mbox{ with } D(F,\KK) = \inf_{\stackrel{(x_s)_{s \in F} \in [0,1]^F}{\sum x_s = 1}} \sup_{\stackrel{G \in \KK}{G \subset F}} \sum_{s \in G} x_s
$$
for $F$ a finite set and $\KK$ a finite simplicial complex or, equivalently, of $D(F,\KK)$. Without loss of generality we can assume $\bigcup \KK \subset F$, so that
$G \in \KK \Rightarrow G \subset F$. For such a $G \subset F$ we denote $\varphi_G$ the linear form on $\R^F$ given by $\underline{x} = (x_s)_s \mapsto \sum_{s \in G} x_s$. We denote $e_s^* = \varphi(\{ s\})$ for $s \in F$.

Since $F$ has only a finite number of subsets, the map
$\underline{x} \mapsto \sum_{G \in \KK} \varphi_G(\underline{x})$ is continuous on $\R^F$, hence its supremum on the compact set $\Delta_F = \{ \underline{x} \in [0,1]^F \ | \ \sum_{s \in F} x_s = 1 \}$ is equal to
$\varphi_{G_1}(\underline{x}_1)$ for some $G_1 \in \KK$ and $\underline{x}_1 \in \Delta_F$. By definition of the supremum, we have $\varphi_{G_1}(\underline{x}_1) \geq
\varphi_{H}(\underline{x}_1)$ for all $H \in \KK$.

For such a $G \in \KK$, let us denote $\mathcal{L}(G)$
the set of all linear forms
$\varphi_G - \varphi_H \in (\R^F)^*$ with $H \in \KK\setminus \{ G \}$. We denote the union of $\mathcal{L}(G)$ with $\{ e_s^*, s \in F \}$ and the set of affine forms $\{ \underline{x} \mapsto 1 - x_s ; s \in F \}$.
Then one checks that 
$$
\varphi_{G_1}(\underline{x}_1) = \inf\lbrace
\varphi_{G_1}(\underline{x}) ; \sum_s x_s = 1\  \& \ \forall\varphi \in \mathcal{L}^+(G_1)\ \   \varphi(\underline{x}) \geq 0 \}
$$
therefore we need to compute the infimum of $\varphi_{G_1}$ on the compact convex set
$C = \{ \underline{x}) ; \sum_s x_s = 1\  \& \ \forall\varphi \in \mathcal{L}^+(G_1)\ \   \varphi(\underline{x}) \geq 0 \}$. By the Krein-Millman (or Minkowski) theorem, $C = C(G_1)$ is equal to the convex hull
inside $\R^F$ of the collection $\mathcal{Y}(G_1)$ of its extremal points. Since $\varphi_{G_1}$ is linear this implies that
$$
\varphi_{G_1}(\underline{x}_1) = \max_{G \in \KK} \sup \{ \varphi_G(\underline{x}) ; \underline{x} \in C(G) \}
=\max_{G \in \KK} \sup \{ \varphi_G(\underline{x}) ; \underline{x} \in \mathcal{Y}(G) \}
$$
Now, each $C(G)$ is the intersection of a finite number of half-spaces whose equations have the form $psi(\underline{x}) \geq 0$ with $\psi$ an affine form with rational coefficients. 
Therefore its set of extremal points is the union of
the \emph{points} $z$ such that $\{z \}$ is the intersection of a subset of the collection of affine hyperplanes of the form $\psi = 0$ with $\psi$ as above.
Therefore it is sufficient to determine the collections of
affine forms inside $L(G)$ with the property that
the intersection of their hyperplanes is a single point
to determine $\mathcal{Y}(G)$. There is a finite number of them. Moreover, since these equations have rational coefficients, the solution of
the corresponding linear system has rational coefficients. It follows that $\mathcal{Y}(G) \subset \Q^F$ for all $G$. In particular $\varphi_{G_1}(\underline{x}_1) = \varphi_{G_1}(\underline{x}_2)$ for some
$\underline{x}_2 \in \mathcal{Y}(G_1) \subset \Q^F$
hence $D(F,\KK) = \varphi_{G_1}(\underline{x}_1) \in \varphi_{G_1}(\Q^F) = \Q$. Moreover, the description above readily provides an algorithm for computing it : determine $\mathcal{Y}(G)$ for all $G \in \KK$ by solving the linear systems attached to $C(G)$, compute $\varphi_G$ on them, take the maximum on $\mathcal{Y}(G)$, and then take the minimum over all $G \in \KK$.

\subsubsection{Reduction to connected components}

\begin{lemma}
Let $\KK$ be a simplicial complex and $F \in \PF^*(S)$.
If $F \not\subset \bigcup \KK$ then $D(F,\KK) = 0$, that is $d(F,\KK)=1$.
\end{lemma}
\begin{proof}
Letting $t \in F \setminus \bigcup \KK$, we define
$\underline{x}^0\in \Delta(F)$ by $x^0_t = 1$ and $x^0_s = 0$ for $s \neq t$. Then
$$
0 \leqslant D(F,\KK) = \inf_{\underline{x} \in \Delta(F)} \sup_{\stackrel{G \in \KK}{G \subset F}}\varphi_G(\underline{x}) \leqslant
\sup_{\stackrel{G \in \KK}{G \subset F}}\varphi_G(\underline{x}^0) = 0
$$
and this proves the claim.
\end{proof}

\begin{lemma} 
\label{lem:opticonnexe}
Let $\Delta = \{ (a_1,\dots,a_r)\in [0,1]^r \ | \ a_1+\dots+a_r = 1\}$ and $u_1,\dots,u_r > 0$. Then
$$
\inf_{\underline{a} \in \Delta} \max(a_1u_1,\dots,a_r u_r)
= \frac{1}{\frac{1}{u_1}+\frac{1}{u_2}+\dots+\frac{1}{u_r}}
$$
\end{lemma}
\begin{proof}
Let $R$ denote the RHS of the equation and $L$ its LHS. Let us consider $\underline{a} \in \R_+^r$ with $a_i = u_i^{-1}R$.
Note that $a_i = \frac{1}{\sum_{j} \frac{u_i}{u_j}} \leq \frac{1}{\frac{u_i}{u_i}} =1$ and $a_1+\dots+a_r = 1$ hence $\underline{a}\in\Delta$. We have $\max(a_1u_1,\dots,a_r u_r) = R$, hence $L \leq R$.

Now let us consider an arbitrary $\underline{a} \in \Delta_r$. If, for any $i \in \{1,\dots,r\}$ we have
$a_i \geq R u_i^{-1}$, then $\max(a_1u_1,\dots,a_r u_r) \geq R$. We claim that it is always the case. For otherwise, we would have $a_i < R u_i^{-1}$ for all $i$,
whence $\sum_i a_i < R \sum_i u_i^{-1}=1$, contradicting $\underline{a}\in \Delta$. This proves $L \geq R$ hence $L = R$. 
\end{proof}

\begin{proposition}
Let $\KK$ be a simplicial complex, $F \in \PF^*(S)$
and $(\KK_i)_{i \in I}$ the connected components of $\KK$.
Then
$$
\frac{1}{D(F,\KK)} = \sum_{i \in I}\frac{1}{D(F_i,\KK_i)}
$$
\end{proposition}
\begin{proof}
First note that, if $F \not\subset \bigcup\KK$,
then $D(F,\KK) =D(F,\KK_i) = 0$ and the claim holds.
Therefore we can assume $F \subset \bigcup\KK$.
Let then $F_i = F \cap \bigcup \KK_i$. The $F_i$ form
a partition of $F$. Let $I_0 = \{ i \in I; F_i \neq \emptyset\}$ and $\KK' = \bigcup_{i \in I_0} \KK_i$.
From the definition one gets immediately $D(F,\KK)=D(F,\KK')$, so we can assume $I_0 = I$, that is $\forall i \in I \ F_i \neq \emptyset$. In particular $I$ can be assumed to be finite.

For $G$ a finite set, we denote $\Delta^{\circ}(G) \subset \Delta(G)$ the set of $\alpha : G \to ]0,1]$ such
that $\sum_{ s \in G} \alpha(s) = 1$. It is a dense subset of $\Delta(G)$, and in particular
$$
D(F,\KK) =\inf_{\underline{x}\in\Delta(F)} \sum_{\stackrel{G \in \KK}{G\subset F}} \varphi_G(\underline{x}) = 
\inf_{\underline{x}\in\Delta^{\circ}(F)} \sum_{\stackrel{G \in \KK}{G\subset F}} \varphi_G(\underline{x})
$$

For $\underline{x} = (x_s)_{s \in F} \in \Delta(F)$, we define
$\alpha(\underline{x}) \in \Delta(I)$ by
$\alpha(\underline{x})_i = \sum_{s \in F_i} x_s$. 
Note that $\underline{x} \in \Delta^{\circ}(F) \Rightarrow \alpha(\underline{x}) \in \Delta^{\circ}(I)$.
Then
$$
D(F,\KK) = \inf_{\underline{x}\in\Delta^{\circ}(F)} \sum_{\stackrel{G \in \KK}{G\subset F}} \varphi_G(\underline{x})
= \inf_{\underline{\alpha}\in\Delta^{\circ}(I)}
\inf_{
\stackrel{
\underline{x}\in\Delta^{\circ}(F)}{
\alpha(\underline{x})=\underline{\alpha}}} \max_{\stackrel{G \in \KK}{G\subset F}} \varphi_G(\underline{x})
= \inf_{\underline{\alpha}\in\Delta^{\circ}(I)}
\inf_{
\stackrel{
\underline{x}\in\Delta^{\circ}(F)}{
\alpha(\underline{x})=\underline{\alpha}}} \max_{i \in I}\max_{\stackrel{G \in \KK_i}{G\subset F_i}} \varphi_G(\underline{x})
$$
Let us denote $I = \{1,\dots,r\}$. Now, for fixed $\underline{\alpha}\in \Delta^{\circ}(I)$ 
we have a bijection
$$\Delta^{\circ}(F_1)\times \dots \times \Delta^{\circ}(F_r) \to \{\underline{x}\in \Delta^{\circ}(F) \ \mid \ \alpha(\underline{x}) =\underline{\alpha}\}$$
given by $(\underline{x}^{(1)},\dots,\underline{x}^{(r)}) \mapsto \underline{x} = (x_s)_{s \in F}$
such that $x_s = x^{(i)}_s \alpha_i$. Moreover, under
this bijection, if $G\subset \KK_i$, we have $\varphi_G(\underline{x}) = \alpha_i \varphi_G(\underline{x}^{(i)})$. It follows that
$$
D(F,\KK) = 
\inf_{\underline{\alpha}\in\Delta^{\circ}(I)}
\inf_{
(\underline{x}^{(1)},\dots,\underline{x}^{(r)}) \in \Delta^{\circ}(F_1)\times \dots\times \Delta^{\circ}(F_r)}
 \max_{i \in I}\max_{\stackrel{G \in \KK_i}{G\subset F_i}} \alpha_i \varphi_G(\underline{x}^{(i)})
=
\inf_{\underline{\alpha}\in\Delta^{\circ}(I)}
\max_{i \in I} \inf_{
\underline{x}^{(i)} \in \Delta^{\circ}(F_i) }
\max_{\stackrel{G \in \KK_i}{G\subset F_i}}\alpha_i \varphi_G(\underline{x}^{(i)})
$$
hence
$$D(F,\KK) = 
\inf_{\underline{\alpha}\in\Delta(I)}
\max_{i \in I} \alpha_i \inf_{
\underline{x}^{(i)} \in \Delta(F_i) }
\max_{\stackrel{G \in \KK_i}{G\subset F_i}} \varphi_G(\underline{x}^{(i)})
 = 
\inf_{\underline{\alpha}\in\Delta(I)}
\max_{i \in I} \alpha_i 
D(F_i,\KK_i)
$$
and by lemma \ref{lem:opticonnexe} this proves the claim.

\end{proof}

\subsection{Special cases}
\label{sect:specialcases}
We now compute a few basic examples.

\begin{enumerate}
\item Assume $S = \N^* = \Z_{>0}$,
and let $\mathcal{K}_n = \mathcal{P}^*(\{1,\dots,n \})$, $\mathcal{K}'_n = \{ \{ i \}; i \in \{1,\dots, n \} \}$.
We have $\mathcal{K}_n,\mathcal{K}'_n \in \mathcal{S}$, $\delta(\mathcal{K}'_n,\mathcal{K}_n) = 0$
and $\delta(\mathcal{K}_n,\mathcal{K}'_n) = (n-1)/n$.
\item Also let us consider $\Delta(S) = \mathcal{P}_f^*(S) \subset  \mathcal{P}_f^*(S)$ considered as a simplicial complex,
and assume $S = \N^*$. Let $\mathcal{K}_n = \Delta(\N^*) \setminus \{ F \in \mathcal{P}_f^*(S) \ | \ F \supset \{1,\dots, n \} \}$.
By proposition \ref{prop:uplowbound} we have $d(\Delta(\N^*),\mathcal{K}_n) \geq 1/n$.
On the other hand, for $F \in \Delta(\N^*)$ we have by lemma \ref{lem:calcdFK} that
$$d(F,\mathcal{K}_n) = \sup_{f_0 (\Omega) = F} \min_{\stackrel{H < F}{H \in \mathcal{K}_n}} \la(\Omega \setminus f_0^{-1}(H))
$$
and we restrict ourselves to considering the maximal $H$. If $F = f_0(\Omega) \in \mathcal{K}_n$ this minimum is $0$,
and otherwise $\{1,\dots,n \} \subset F \subset f_0(\Omega)$. In this case, for any $H < F$ with $H \in \mathcal{K}_n$ there exists $h_0 \in \{1,\dots, n \} \setminus H$
and $H$ is contained inside $F \setminus \{ h_0\}$ which belongs to $\mathcal{K}_n$ and has cardinality $|F|-1$. Therefore,
when $F \not\in \mathcal{K}_n$ we have
$$
d(F,\mathcal{K}_n) = \sup_{f_0 (\Omega) = F} \min_{i \in \{1,\dots n \}} \la(\Omega \setminus f_0^{-1}(F \setminus \{ i \})) \leqslant \frac{1}{n}
$$
since $\Omega \setminus f_0^{-1}(F \setminus \{ i \})
 = f_0^{-1}(\{ i \})$ and $1 = \sum_{i=1}^n \la(f_0^{-1}(\{ i \}))$. It follows that 
$d(\Delta(\N^*), \mathcal{K}_n) \leq 1/n$ and $d(\Delta(\N^*), \mathcal{K}_n) \to 0$.
\item For $X$ a set, let us denote $\mathcal{P}^*_{\leq k}(X) = \{ E \in \mathcal{P}_f^*(X) \ | \ |E| \leq k \}$. It is a simplicial complex
over $X$. We compute the distance between the simplicial complexes $\mathcal{P}^*_f(X)$ and $\mathcal{P}^*_{\leq k}(X)$.
Clearly $\delta(\mathcal{P}^*_{\leq k}(X),\mathcal{P}^*_f(X)) = 0$. We have
$$
\delta(\mathcal{P}^*_f(X),\mathcal{P}^*_{\leq k}(X)) = \sup_{F \in \mathcal{P}_f^*(X) } d(F,\mathcal{P}^*_{\leq k}(X))
$$
and, for any finite $F \subset X$, $d(F,\mathcal{P}^*_{\leq k}(X))$ is either $0$ (when $|F| \leq k$) or the sup of
the $\la(\Omega \setminus f^{-1}(H))$ for $H \subset F$ with cardinality $k$. By proposition \ref{prop:uplowbound} we
get $d(F,\mathcal{P}^*_{\leq k}(X)) = 1 - k/|F|$. The supremum over all such $F$ is then either $1$ if $X$ is infinite, or
$1 - k/|X|$. With the convention $1/\infty = 0$ this yields
$$
d(\mathcal{P}^*_f(X),\mathcal{P}^*_{\leq k}(X))  = 1 - \frac{k}{|X|}
$$
Note that the first example is a special case of this one.
\item Let us assume that $S$ is an infinite set. Let $u_n = p_n/q_n, n \geq 1$ be a sequence of positive rational numbers, with $p_n,q_n \in \N^*$ and $p_n \leq q_n$. Let us consider an infinite (countable) subset of $S$ partitioned as
$S_1\sqcup S_2 \sqcup \dots$ with $|S_n| = q_n$. We consider the following two simplicial complexes over $S$.
$$
\mathcal{K}_2 = \bigcup_n \mathcal{P}_f^*(S_n) \ \ \mathcal{K}_1 = \bigcup_n \mathcal{P}^*_{\leq q_n-p_n}(S_n)
$$
Since $\mathcal{K}_1 \subset \mathcal{K}_2$ we have $\delta(\mathcal{K}_1,\mathcal{K}_2) = 0$. Now, since any $F \in \mathcal{K}_2$
belongs to some $\mathcal{P}_f^*(S_n)$, we have $d(F,\mathcal{K}_1) = d(F,\mathcal{P}^*_{\leq q_n-p_n}(S_n))$
hence
$$
d(\mathcal{K}_1,\mathcal{K}_2) = \sup_n d\left(\mathcal{P}_f^*(S_n),\mathcal{P}^*_{\leq q_n-p_n}(S_n)\right) = \sup_n 1 - \frac{q_n - p_n}{q_n} = \sup_n \frac{p_n}{q_n}
$$
\end{enumerate}

\medskip

\section{Topology of $\mathcal{S}$}

\begin{proposition} \label{prop:SndiscrSinfnotdense} For all $n$, $\mathcal{S}_n$ is discrete. When $S$ is infinite, $\mathcal{S}_{\infty} = \bigcup_n \mathcal{S}_n$ is \emph{not} dense inside $\mathcal{S}$.
More specifically, $\mathcal{P}_f^*(S)$ considered as a simplicial complex does not belong to the
closure of $\mathcal{S}_{\infty}$, and its distance to it is equal to $1$. Finally, when $S$ is infinite (and countable ?) then $\mathcal{S}_{\infty})$ is equal to the set of isolated points of $\mathcal{S}$.
\end{proposition}
\begin{proof}
For all $\mathcal{K}_1,\mathcal{K}_2 \in \mathcal{S}_n$ with $\mathcal{K}_1\neq \mathcal{K}_2$, we have
by proposition \ref{prop:uplowbound} that
$d(\mathcal{K}_1,\mathcal{K}_2) \geq 1/(N+2)$ for some $N$ with $N+1 \leq n$,
whence $d(\mathcal{K}_1,\mathcal{K}_2) \geq 1/(n+1)$. This proves that $\mathcal{S}_n$
is discrete.

Let $\mathcal{K} \in \mathcal{S}_n$. Since $\mathcal{K} \subset \mathcal{P}_f^*(S)$
we have, for any $F \not\in \mathcal{K}$ and any maximal subset $H$ of $F$ belonging to $\mathcal{K}$, that
$$
d(\mathcal{K},\mathcal{P}_f^*(S)) = \delta(\mathcal{P}_f^*(S),\mathcal{K}) \geqslant
d(F,\mathcal{K}) \geqslant 1 - \frac{|H|}{|F|}\geqslant 1 - \frac{n+1}{|F|}
$$
by proposition \ref{prop:uplowbound}.
Since $|F|$ can be chosen arbitrarily large this proves $d(\mathcal{K},\mathcal{P}_f^*(S)) = 1$. Since this holds
true for every $\mathcal{K} \in \mathcal{S}_n$ this proves the second claim.

Let now $\mathcal{K}_0 \in \mathcal{S}_n$. For any $\mathcal{K} \in \mathcal{S}$, if $\mathcal{K}\not\subset \mathcal{K}_0$, there exists $F \in \mathcal{K} \setminus \mathcal{K}_0$ with $|F| \leq n+2$. Then $d(\mathcal{K},\mathcal{K}_0) \geq \delta(\mathcal{K},\mathcal{K}_0) \geq d(F,\mathcal{K}_0) \geq 1/(n+2)$ by proposition \ref{prop:uplowbound}. If $\mathcal{K} \subset \mathcal{K}_n$ then $\mathcal{K} \in \mathcal{S}_n$ and we already
know $d(\mathcal{K},\mathcal{K}_0) \geq 1/(n+1)$.
It follows that $d(\mathcal{K},\mathcal{K}_0) \geq 1/(n+2)$
in all cases and this proves that all elements of $\mathcal{S}_{\infty}= \bigcup_n \mathcal{S}_n$ are isolated points in $\mathcal{S}$. Conversely, let $\mathcal{K} \not\in \mathcal{S}_{\infty}$ and let $n > 0$. We set $X = \bigcup \mathcal{K} \subset S$. Since $\mathcal{K} \not\in \mathcal{S}_{\infty}$ then $X$ contains a countable subset $X_0$ containing a face of $\mathcal{K}$ of cardinality $n$.
Up to relabelling, 
we can assume $X_0 =  \N^*$ and $\{1,\dots, n \} \in \mathcal{K}$. Let $\Delta(\N^*) = \mathcal{P}_f^*(\N^*) \in \mathcal{S}$, and $\Delta_n(\N^*) = \Delta(\N^*) \setminus \{ F \in \mathcal{P}_f^*(\N^*) \ | \ F \supset \{ 1,\dots, n \} \}$. We then let $\mathcal{K}_n = \mathcal{K} \cap \Delta_n(\N^ *)$. Then, by applying twice proposition \ref{prop:intersectcontract}, we get
$$
d(\mathcal{K},\mathcal{K}_n) \leqslant
d(\mathcal{K}\cap \Delta(\N^*),\mathcal{K}_n\cap \Delta(\N^*)) =
d(\mathcal{K}\cap \Delta(\N^*),\mathcal{K} \cap \Delta_n(\N^*)) \leqslant
d( \Delta(\N^*), \Delta_n(\N^*))
$$ 
and we know by section \ref{sect:specialcases} example (2) that 
$d( \Delta(\N^*), \Delta_n(\N^*)) \leq 1/n$.
Since $\{ 1 , \dots, n \} \in \mathcal{K} \setminus \mathcal{K}_n \}$ this proves that $0 < 
d(\mathcal{K},\mathcal{K}_n) \leq 1/n$. This proves that $\mathcal{K}$ is an accumulation point of $\mathcal{S}$ and
this concludes the proof of the proposition.
\end{proof}

As a corollary of proposition \ref{prop:uplowbound} we get the following.
\begin{proposition}
$\mathcal{S}$ is \emph{not} locally compact.
\end{proposition}
\begin{proof}
In order to prove this, it is enough to find $\mathcal{K}_0 \in \mathcal{S}$ such that, for every $\eps>0$,
the closed ball centered at $\mathcal{K}_0$ with radius $\eps$ is \emph{not} compact. Let
us consider $\mathcal{K}_0 \in \mathcal{S}$ having cells of every dimension, and $\eps>0$. Let us choose
$N_0$ with $N_0+2 \geq \eps$. Every $\mathcal{K}'$ with $\pi_{N_0}(\mathcal{K}')= \pi_{N_0}(\mathcal{K}_0)$
will satisfy $d(\mathcal{K}_0,\mathcal{K}') \leq 1/(N_0+2) \leq \eps$ and therefore belong to the ball.
Let us set $\mathcal{K}_{00} = \pi_{N_0}(\mathcal{K}_0)$. Notice that
$\pi_{1+N_0}(\mathcal{K}_0) \setminus \pi_{N_0}(\mathcal{K}_0)$ is infinite,
for otherwise $\mathcal{K}_0$ would be finite, contradicting our assumption. We denote
$\pi_{1+N_0}(\mathcal{K}_0) \setminus \pi_{N_0}(\mathcal{K}_0) = \{ F_1,F_2,\dots \}$ and set
$\mathcal{K}_r = \mathcal{K}_{00} \cup \{ F_1,\dots,F_r \}$. It is a simplicial complex,
and we have $d(\mathcal{K}_r,\mathcal{K}_s) \geq 1/(2+N_0)$ by the above, hence it does not
admit any accumulation point. This proves that the neighborhood is not compact and the claim.
\end{proof}
Let $\mathcal{S}_f \subset \mathcal{S}$ denote the collection of \emph{finite} simplicial complexes. For
$\mathcal{K} \in \mathcal{S}$ we have $\mathcal{K} \in \mathcal{S}_f$ iff $\exists E \subset S$
finite such that $\mathcal{K} \subset \mathcal{P}(E)$ iff $\pi_0(\mathcal{K})$ is finite.

\begin{proposition} $\mathcal{S}_f$ is discrete and, for all $\mathcal{K} \in \mathcal{S} \setminus \mathcal{S}_f$, we have $d(\mathcal{K},\mathcal{S}_f) = 1$.
\end{proposition}
\begin{proof}
Let $\mathcal{K} \in \mathcal{S}_f$, and $K = \pi_0(\mathcal{K}) \subset S$. We want to find $\eps > 0$ such that
$d(\mathcal{K}',\mathcal{K}) < \eps$ and $\mathcal{K}' \in \mathcal{S}_f$ implies $\mathcal{K}'=\mathcal{K}$.
Note that, if $\pi_0(\mathcal{K}' ) \neq K$, then $d(\mathcal{K}',\mathcal{K}) \geq d(\pi_0(\mathcal{K}'),\pi_0(\mathcal{K})) =1$.
By assuming $\eps < 0$ we can thus assume $\pi_0(\mathcal{K}')=K$. But there exists only a finite number of simplicial
complexes with $\pi_0(\mathcal{K}')=K$ hence letting $\eps = \min(1/2, \min \{ d(\mathcal{K},\mathcal{K}')/2; \pi_0( \mathcal{K}') = \mathcal{K} \})$
we get what we want and $\mathcal{S}_f$ is discrete.

For the second claim, let $\mathcal{K}_1 \in  \mathcal{S} \setminus \mathcal{S}_f$ and $\mathcal{K}_2 \in  \mathcal{S}_f$. We have
$d(\mathcal{K}_1,\mathcal{K}_2) \geq d(F,\mathcal{K}_2)$ for any $F \in \mathcal{K}_1$. We have
$d(F,\mathcal{K}_2) \geq 1 - |H|/|F|$ where $H < F$ has maximal cardinality with $H \in \mathcal{K}_2$. Since $\mathcal{K}_2$
is finite, letting $h$ equal to the maximal cardinality of its elements we get $d(F,\mathcal{K}_2) \geq 1 - |H|/|F|$. But since $F$ can be chosen
as large as needed, this proves $d(\mathcal{K}_1,\mathcal{K}_2) = 1$ and the claim.

\end{proof}

In particular, $\mathcal{S}_f$ is \emph{not} dense inside $\mathcal{S}$.

\medskip

We let $\mathbf{Met}_1$ denote the category of all the metric spaces with diameter at most $1$ and $1$-Lipschitz maps.
This category admits arbitrary limits (see e.g. \cite{CCS} proposition 4.7) and any inverse limit of complete metric spaces
is complete (\cite{CCS} proposition 4.10). By the explicit construction of the inverse limit in this category (\cite{CCS} prop. 4.7) the result above
justifies the first part of the following claim.

\begin{proposition} \label{prop:Sinvlimit} Let us consider the inverse directed system of the  $(\mathcal{S}_n)_{n \geq 0}$ with \emph{surjective} transition
maps $(\pi_{n})_{|\mathcal{S}_n} : \mathcal{S}_n \to \mathcal{S}_m$. Its inverse limit inside $\mathbf{Met}_1$ is naturally identified
with $\mathcal{S}$, and the underlying topology of $\mathcal{S}$ is (strictly) finer than the topology of the inverse limit inside $\mathbf{Top}$ of (the directed system made of the underlying topological spaces of) the $(\mathcal{S}_n)_{n \geq 0}$. The space
$\mathcal{S}$ is not discrete, and is totally disconnected.
\end{proposition}
\begin{proof} By construction (see \cite{CCS} prop. 4.7) the inverse limit in $\mathbf{Met}_1$ can be constructed as the obvious
subspace of $\hat{\mathcal{S}} = \prod_n \mathcal{S}_n$ with metric the supremum of the metrics of the projection. By proposition \ref{prop:pinlip}
this is naturally identified with $\mathcal{S}$. 

In order to prove the second claim we now identify $\mathcal{S}$ with the subspace of $\hat{\mathcal{S}} = \prod_n \mathcal{S}_n$
described by the transition maps, and denote $\mathcal{S}_{\pi}$ the same subset by endowed with the
product topology. Since the projection maps are $1$-Lipschitz hence continuous, we get that the
natural map $\mathcal{S} \to \mathcal{S}_{\pi}$ is continuous and this proves that the metric topology is finer. The fact
that these topologies are not the same is a consequence of the density of $\mathcal{S}_{\infty}$ inside $\mathcal{S}_{\pi}$, while it is not inside $\mathcal{S}$ by proposition \ref{prop:SndiscrSinfnotdense}.

The fact that the topology of $\mathcal{S}$ is not discrete is a consequence of the fact that
we constructed a sequence $\mathcal{K}_n$ of simplicial complexes such that $\mathcal{K}_n \to \Delta(S)$
and $d(\mathcal{K}_n,\Delta(S)) \geq 1/n$. The fact that it is totally disconnected is a consequence
of the fact that its topology is finer that the pro-discrete topopogy of $\mathcal{S}_{\pi}$ which is itself
totally discontinuous (alternatively: if $x\neq y$ belong to a connected subset $C$ of $\mathcal{S}$, 
then $\pi_n(C)$ is connected inside the discrete space $\mathcal{S}_n$ hence $\pi_n(x) = \pi_n(y)$
for all $n$ whence $x = y$).

\end{proof}

The topology of $\mathcal{S}$ given by the topological inverse limit of the $\mathcal{S}_n$ will be called in
the sequel its \emph{pro-discrete topology}.

\begin{corollary} $\mathcal{S}$ is a complete metric space.
\end{corollary}
\begin{proof} Since the metric spaces $\mathcal{S}_n$ are discrete by proposition \ref{prop:SndiscrSinfnotdense}, they are complete, therefore their limit insite $\mathbf{Met}_1$
is complete (\cite{CCS} proposition 4.10), and this proves the claim.
\end{proof}

We end this section by noticing that the standard operation of barycentric subdivision is very brutal with respect to the Hausdorff metric. First recall, that the barycentric
subdivision $sd(\KK) \subset \PF^*(\KK)$ of
the simplicial complex $\KK \subset \PF^*(S)$ is classically defined as
$$
sd(\KK) = \{ \{ x_1,\dots,x_r \} \in \PF(\KK) \ | r \geq 1, \ x_1 \subsetneq x_2 \subsetneq \dots \subsetneq x_r \}
$$
and that it is a simplicial complex. Using the set-theoreric notation $\bigcup \{ x_1,\dots, x_r \} = x_1 \cup \dots \cup x_r$,
we have that, for $X \neq \emptyset$,  $X \in sd(\KK) \Leftrightarrow \bigcup X \in \KK$.
The map $X \mapsto \bigcup X$ maps the simplices of $sd(\mathcal{K})$ to simplices of $\mathcal{K}$.
One then gets the following.
\begin{proposition}
If $\mathcal{K}_1,\mathcal{K}_2 \subset \PF^*(S)$ are two simplicial complexes,
with $\KK_1 \neq \KK_2$ then $d(sd(\KK_1),sd(\KK_2)) = 1$.
\end{proposition}
\begin{proof} By symmetry we can assume there exists $F \in \KK_1 \setminus \KK_2$.
Then, for every $\mathcal{F} \in sd(\KK_2)$, we have $F \not\in \mathcal{F}$.
Therefore, if $f$ is the constant map equal to $F$, we have $d(f,\KK_2) = 1$
whence $\delta(\KK_1,\KK_2) = 1$ and $d(\KK_1,\KK_2) = 1$. 
\end{proof}

\section{Hausdorff metric on isomorphism classes}

\subsection{Definition and independence w.r.t. the ambient vertex set}
Let $K_1$, $K_2$ be two isomorphism classes of simplicial complexes. This
means a priori that there a fixed set $S$ is chosen and that $K_1,K_2$ are equivalence classes
on the collection of simplicial complexes with vertices belonging to $S$, where
the equivalence relation is $\mathcal{K}_1 \sim  \mathcal{K}_2$ if there
exists $\sigma \in \mathfrak{S}(S)$ such that $F \mapsto \sigma(F)$
is a bijection $\mathcal{K}_1 \to \mathcal{K}_2$. Therefore this notion
a priori depends on the choice of an ambient vertex set $S$.

We define a distance function on the set $\mathcal{I}(S)$ of such isomorphism classes
by
$$
d_S(K_1,K_2) = \inf_{\stackrel{\mathcal{K}_1 \in K_1}{\mathcal{K}_2 \in K_2}}
d_S(\mathcal{K}_1,\mathcal{K}_2)
$$
where we denote $d_S(\mathcal{K}_1,\mathcal{K}_2)$ the distance
$d(\mathcal{K}_1,\mathcal{K}_2)$ precedently defined, where the set $S$ was previously understood.
For any $\mathcal{K}_1 \in K_1$ and $\mathcal{K}_2 \in K_2$, one immediately gets that
$$
d_S(K_1,K_2) = \inf_{\sigma_1,\sigma_2 \in \mathfrak{S}(S)}
d(\sigma_1(\mathcal{K}_1),\sigma_2(\mathcal{K}_2))
= \inf_{\sigma \in \mathfrak{S}(S)}
d_S(\mathcal{K}_1,\sigma(\mathcal{K}_2))
$$
and this easily implies that $d_S$ is indeed a distance function on $\mathcal{I}(S)$. 

When $S \subset T$, we identify $\mathfrak{S}(S)$ with the subgroup of $\mathfrak{S}(T)$
made of the bijections which are the identity on $T \setminus S$. If $\mathcal{K}_1,\mathcal{K}_2$ are simplicial complexes over $S$, then they are isomorphic as simplicial complexes over $S$ iff they are isomorphic as simplicial complexes over $T$. Therefore $\mathcal{I}(S)$ is naturally identified with a subset of $\mathcal{I}(T)$. 
We first check that.
\begin{proposition} \label{prop:indepST}
If $S \subset T$ and $K_1,K_2 \in \mathcal{I}(S)$, then
$$
d_S(K_1,K_2) = \inf_{\sigma \in \mathfrak{S}(T)}
d_T(\mathcal{K}_1,\sigma(\mathcal{K}_2))
$$
\end{proposition}
\begin{proof}

Let $K_1, K_2 \in \mathcal{I}(S) \subset \mathcal{I}(T)$, and $\mathcal{K}_1^S,
\mathcal{K}_2^S$  be representatives of $K_1,K_2$ in $\mathcal{S}$.
We have
$$
d_T(K_1,K_2) = \inf_{\sigma \in \mathfrak{S}(T)} d_T(\sigma \KK_1^S, \KK_2^S) \leqslant 
\inf_{\sigma \in \mathfrak{S}(S)} d_T(\sigma \KK_1^S, \KK_2^S)
=
\inf_{\sigma \in \mathfrak{S}(S)} d_S(\sigma \KK_1^S, \KK_2^S) = d_S(K_1,K_2)
$$
where we use the shortcut $\sigma \mathcal{K}$ for $\sigma(\mathcal{K})$.
We need to prove
$$
\inf_{\sigma \in \mathfrak{S}(T)} d_T(\sigma \KK_1^S, \KK_2^S) \geqslant 
\inf_{\sigma \in \mathfrak{S}(S)} d_T(\sigma \KK_1^S, \KK_2^S)
$$
and by symmetry it is sufficient to prove
$$
\inf_{\sigma \in \mathfrak{S}(T)} \delta_T(\sigma \KK_1^S, \KK_2^S) \geqslant 
\inf_{\sigma \in \mathfrak{S}(S)} \delta_T(\sigma \KK_1^S, \KK_2^S)
$$
Let $\eps > 0$. There is $\sigma_0 \in \mathfrak{S}(T)$ such that
$\delta_T(\sigma_0 \KK_1,\KK_2) \leq \inf_{\sigma \in \mathfrak{S}(T)} \delta_T(\sigma \KK_1^S, \KK_2^S) + \eps$.
Now  
$$\delta_T(\sigma_0 \KK_1,\KK_2) = \sup_{f_0(\Omega) \in \sigma_0 \KK_1} d_T(f_0,L(\KK_2))
$$
Now, for any given such $f_0$ such that $f_0(\Omega) \in \sigma_0 \KK_1$,
either $d_T(f_0,L(\KK_2)) = 1$, or $f_0(\Omega) \cap S \neq\emptyset$.
In the latter case we can assume that $f_0(\Omega) = \{ \sigma_0(x_1),\dots,\sigma_0(x_r)\}$ with $\{x_1,\dots,x_r \} \in \KK_1$ and $\sigma_0(x_1),\dots,\sigma_0(x_i) \in S$ with
$i \geq 1$. Let $f'_0(t) = f_0(t)$ it $f_0(t) \in S$ and $f'_0(t) = \sigma_0(x_1)$
otherwise. We have $f'_0(\Omega) \subset f_0(\Omega) \in \sigma_0 \KK_1$
hence $f'_0(\Omega) \in \sigma_0 \KK_1$. Note that $x_k, \sigma_0(x_k) \in S$ for $k \leq i$. Let us choose $\sigma'_0 \in \mathfrak{S}(S)$ such that $\sigma'_0(x_k) = \sigma_0(x_k)$ for $k \leq i$. Then $d_T(f_0,L(\KK_2)) \geq d_T(f'_0,L(\KK_2))
= d_S(f'_0,L(\KK_2))$ and $f'_0(\Omega) \in \sigma'_0 \KK_1$. It follows
that 
$$
\sup_{f_0(\Omega) \in \sigma_0 \KK_1} d_T(f_0,L(\KK_2))
\geqslant
\sup_{f'_0(\Omega) \in \sigma'_0 \KK_1} d_S(f'_0,L(\KK_2))
= \delta_T( \sigma'_0 \KK_1,\KK_2)
= \delta_S( \sigma'_0 \KK_1,\KK_2)
$$
hence
$$
\inf_{\sigma \in \mathfrak{S}(T)} \delta_T(\sigma_0 \KK_1,\KK_2) \geqslant
\delta_T(\sigma_0 \KK_1,\KK_2) - \eps \geqslant \delta_S( \sigma'_0 \KK_1,\KK_2) - \eps \geqslant \inf_{\sigma \in \mathfrak{S}(S)} \delta_S(\sigma \KK_1^S, \KK_2^S) - \eps
$$
and since this is true for all $\eps > 0$ this proves the claim.
\end{proof}
Because of this proposition, there is no drawback in removing the set $S$ from the
notation $d_S(K_1,K_2)$, and have it understood as before, also for isomorphism classes of simplicial complexes.

\begin{corollary} \label{cor:findKisoQ} If $K_1,K_2$ are isomorphism classes of finite simplicial complexes, then $d(K_1,K_2) \in \Q$.
\end{corollary}
\begin{proof} Let $\mathcal{K}_1,\KK_2$ be representatives of $K_1,K_2$.
Because of the above proposition we can assume $S = (\bigcup \KK_1)\cup (\bigcup \KK_2)$. But since $S$ is finite $\mathfrak{S}(S)$ is also finite and
$$
\inf_{\sigma \in \mathfrak{S}(S) } d(\sigma \KK_1,\KK_2)
= 
\min_{\sigma \in \mathfrak{S}(S) } d(\sigma \KK_1,\KK_2)
$$
belongs to $\Q$ by proposition \ref{prop:dKKQ}.
\end{proof}
\subsection{Other properties and special computations}

\begin{proposition} \label{prop:findimdd}
If $K_1$ and $K_2$ are finite dimensional, then they admit
representatives $\KK_1,\KK_2$ such that $d(K_1,K_2) = d(\KK_1,\KK_2)$.
\end{proposition}
\begin{proof}
This is an immediate consequence of the fact that $d$ has finite image over
$\mathcal{S}_n \times \mathcal{S}_m$ with $n = \dim K_1$ and $m = \dim K_2$, by
proposition \ref{prop:finSn}.
\end{proof}

\begin{remark}
We do not know whether the conclusion holds in full generality, that is if there
exists simplicial complexes $\KK_1,\KK_2$ with isomorphism classes $K_1,K_2$ such that
$$
\forall \sigma \in \mathfrak{S}(S) \ d(\sigma(\KK_1),\KK_2) > d(K_1,K_2)
$$
\end{remark}

\begin{proposition} \label{cor:ineqsquela}
If the $(n+1)$-squeletons of $K_1$ and $K_2$ differ, then $d(K_1,K_2) \geq \frac{1}{n+2}$
\end{proposition}
\begin{proof}
Let $\mathcal{K}_1,\mathcal{K}_2$ be arbitrary representatives of $K_1,K_2$, respectively. By corollary \ref{cor:ineqsquel} we have $d(\KK_1,\KK_2) \geq \frac{1}{n+2}$. The conclusion follows.
\end{proof}

We now compute the distance for a few basic examples. 

\begin{proposition}{\ } \label{prop:excompiso}
\begin{enumerate}
\item If the cardinality of the vertices of $K_1$ and $K_2$ are not the same, then
$d(K_1,K_2) = 1$.
\item Let $X$ a finite set of cardinality $q \geq 1$, and $1 \leq p \leq q$ an integer. Let $\KK_1 = \mathcal{P}_{\leq q-p}(X)$ ,$\KK_2 = \PF^*(X)$
and $K_1,K_2$ the corresponding isomorphism classes. Then $d(K_1,K_2) = p/q$.
\item Let $(S_n)$ be a collection of disjoint finite sets with $q_n = |S_n|$, and $(p_n)$ an integer sequence with $1 \leq p_n \leq q_n$. Let $\mathcal{K}_1 = \bigsqcup_n \mathcal{P}^*_{\leq q_n-p_n}(S_n)$,   $\mathcal{K}_2 = \bigsqcup_n \PF^*(S_n)$,
and $K_1,K_2$ the corresponding isomorphism classes. Assume that $\alpha = \sup_n p_n/q_n$. If $\alpha \leq 1/2$ and $q_n$ is (strictly) increasing, then $d(K_1,K_2) = \alpha$. 
\end{enumerate}
\end{proposition}
\begin{proof}
(1) is an immediate consequence of corollary \ref{cor:dist1vertex}.
We prove (2). We know from section \ref{sect:specialcases} that $d(\KK_1,\KK_2) = p/q$. If $p = q$ then $K_1 = K_2$ and the claim is clear. Let us assume otherwise. Let $\sigma \in \mathfrak{S}(S)$ and assume that $d(\sigma \KK_1,\KK_2) < d(\KK_1,\KK_2)$. Then $\sigma \KK_1$ and $\KK_2$ need to have the same vertex set, hence $\sigma(X) = X$. But then $\sigma(\KK_1) = \KK_1$ contradicting $d(\sigma \KK_1,\KK_2) < d(\KK_1,\KK_2)$.
Thus $d(K_1,K_2) = d(\KK_1,\KK_2) = p/q$.
We now prove (3).
We prove that, under the assumptions of (3), we have $d(\KK_1,\KK_2) = d(K_1,K_2)$. For otherwise, there would exists $\sigma \in \mathfrak{S}(S)$ such that
$d(\sigma \KK_1 , \KK_2) < d(\KK_1,\KK_2) \leq 1/2$. Let us consider such a $\sigma$. For any $a \neq b$ with $a,b \in S_n$ and $\sigma(a) \in S_m$, we claim
that $\sigma(b) \in S_m$. For otherwise, the 1-squeletton of $\sigma\KK_1$ would contain $\{ \sigma(a),\sigma(b) \}$ which does not belong to $\KK_2$. Thus
$d(\sigma \KK_1,\KK_2) \geq 1/2$ by corollary \ref{cor:ineqsquel}, and this contradicts $d(\sigma \KK_1,\KK_2) < d(\KK_1,\KK_2) \leq 1/2$. It follows from this
that $\sigma$ permutes the $S_n$'s. But since $q_n = |S_n|$ is strictly increasing
and $\sigma$ induces a bijection $S_n \to \sigma(S_n)$ this implies that
$\sigma(S_n) = S_n$ for all $n$. But then $\sigma(\KK_1) = \KK_1$,
contradicting $d(\sigma \KK_1 , \KK_2) < d(\KK_1,\KK_2)$. This contradiction proves the claim.
\end{proof}

\begin{remark}
By the above proposition we know that the diameter of $\mathbf{S}(X)$
is $1$, and that if $X$ is infinite then the distance can take the value of any $\alpha \in [0,1/2]$ and any $\alpha \in \Q \cap [0,1]$. It would be interesting to have a construction with distances $\alpha \in ]1/2,1[ \setminus \Q$.
\end{remark}

We denote $\mathbf{S}(X)$ the space of isomorphism classes of simplicial
complexes over the vertex set $X$, and in particular $\mathbf{S}(n)$ the space
$\mathbf{S}(X)$ for $X$ equal to the ordinal $n$.

The spaces $\mathbf{S}(1)$ has 1 element, $\mathbf{S}(2)$ has 2 elements at distance $1/2$.

\begin{center}
\resizebox{5cm}{!}{
$ \displaystyle
\begin{array}{|c||m{1 cm}|c|c|c|c|}
\cline{2-6}
\multicolumn{1}{c||}{} & 
\begin{tikzpicture}[xscale=0.6,yscale=0.6]
\fill (0,0) circle (0.05);
\fill (1,0) circle (0.05);
\fill (0.5,0.866025) circle (0.05);
\end{tikzpicture} 
& \begin{tikzpicture}[xscale=0.6,yscale=0.6]
\fill (0,0) circle (0.05);
\fill (1,0) circle (0.05);
\fill (0.5,0.866025) circle (0.05);
\draw (0,0) -- (1,0);
\end{tikzpicture}
  & \begin{tikzpicture}[xscale=0.6,yscale=0.6]
\fill (0,0) circle (0.05);
\fill (1,0) circle (0.05);
\fill (0.5,0.866025) circle (0.05);
\draw (0,0) -- (0.5,0.866025) -- (1,0);
\end{tikzpicture} 
  & \begin{array}{c} \ \\
\begin{tikzpicture}[xscale=0.6,yscale=0.6]
\fill (0,0) circle (0.05);
\fill (1,0) circle (0.05);
\fill (0.5,0.866025) circle (0.05);
\draw (0,0) -- (1,0) -- (0.5,0.866025) -- (0,0);
\end{tikzpicture} \\ \ 
\end{array} 
&
\begin{tikzpicture}[xscale=0.6,yscale=0.6]
\fill (0,0) circle (0.05);
\fill (1,0) circle (0.05);
\fill (0.5,0.866025) circle (0.05);
\fill (0,0) -- (1,0) -- (0.5,0.866025);
\end{tikzpicture} 
\\
 \hline
 \hline
\begin{array}{c} \ \\ \begin{tikzpicture}[xscale=0.6,yscale=0.6]
\fill (0,0) circle (0.05);
\fill (1,0) circle (0.05);
\fill (0.5,0.866025) circle (0.05);
\end{tikzpicture} \end{array}
& 0 & \frac{1}{2} & \frac{1}{2} & \frac{1}{2} & \frac{2}{3}  \\ 
\hline
\begin{array}{c} \ \\ \begin{tikzpicture}[xscale=0.6,yscale=0.6]
\fill (0,0) circle (0.05);
\fill (1,0) circle (0.05);
\fill (0.5,0.866025) circle (0.05);
\draw (0,0) -- (1,0);
\end{tikzpicture} \end{array}& \ & 0 & \frac{1}{2} & \frac{1}{2} & \frac{1}{2}  \\
\hline
\begin{array}{c} \ \\
\begin{tikzpicture}[xscale=0.6,yscale=0.6]
\fill (0,0) circle (0.05);
\fill (1,0) circle (0.05);
\fill (0.5,0.866025) circle (0.05);
\draw (0,0) -- (0.5,0.866025) -- (1,0);
\end{tikzpicture} \end{array}& \ & \ & 0 & \frac{1}{2} & \frac{1}{2}  \\
\hline
\begin{array}{c} \ \\
\begin{tikzpicture}[xscale=0.6,yscale=0.6]
\fill (0,0) circle (0.05);
\fill (1,0) circle (0.05);
\fill (0.5,0.866025) circle (0.05);
\draw (0,0) -- (1,0) -- (0.5,0.866025) -- (0,0);
\end{tikzpicture}\end{array} & \ & \ & \ & 0 & \frac{1}{3 }\\
\hline
\begin{array}{c} \ \\
\begin{tikzpicture}[xscale=0.6,yscale=0.6]
\fill (0,0) circle (0.05);
\fill (1,0) circle (0.05);
\fill (0.5,0.866025) circle (0.05);
\fill (0,0) -- (1,0) -- (0.5,0.866025);
\end{tikzpicture}\end{array} & \ & \ & \ & \ & 0 \\
\hline
\end{array}
$
}
\end{center}
In table \ref{tab:simp4}, describing the 20-elements metric space $\mathbf{S}(4)$, the 2-dimensional faces are depicted in red while the (single possible) 3-face is depicted in blue.

The 180-elements metric space $\mathbf{S}(5)$ can similarly be computed. The results can be found at \url{http://www.lamfa.u-picardie.fr/marin/d5haus-en.html}. The distance function on this space
takes for values all the values $0 \leq a/b < 1$ with $b \leq 5$, as well as 
the values $\{ \frac{2}{7},\frac{3}{8}, \frac{3}{7}, \frac{4}{9},\frac{5}{9},\frac{4}{7},\frac{5}{8},\frac{5}{7} \}$.

\begin{table}
\resizebox{14cm}{!}{
$\displaystyle
\begin{array}{|c|c|c|c|c|c|c|c|c|c|c|c|c|c|c|c|c|c|c|c|c|}
\cline{2-21}
\multicolumn{1}{c|}{} & 
 \begin{tikzpicture}
\draw (-0.9659258263,-0.1830127019) node (1)   {$\bullet$};
\draw (0.2588190451,0.6830127019) node (2)   {$\bullet$};
\draw (0.7071067812,-0.5) node (3)   {$\bullet$};
\draw (0,-1.2247448714) node (4)   {$\bullet$};
\draw (1) -- (2);
\draw (1) -- (3);
\draw (1) -- (4);
\draw (2) -- (3);
\draw (2) -- (4);
\draw (3) -- (4);
\end{tikzpicture}
&
 \begin{tikzpicture}
\draw (-0.9659258263,-0.1830127019) node (1)   {$\bullet$};
\draw (0.2588190451,0.6830127019) node (2)   {$\bullet$};
\draw (0.7071067812,-0.5) node (3)   {$\bullet$};
\draw (0,-1.2247448714) node (4)   {$\bullet$};
\draw (1) -- (2);
\draw (1) -- (3);
\draw (1) -- (4);
\draw (2) -- (3);
\draw (2) -- (4);
\draw (3) -- (4);
\draw[ultra thick] (-0.9659258263,-0.1830127019) -- (0.2588190451,0.6830127019);
\end{tikzpicture} &
 \begin{tikzpicture}
\draw (-0.9659258263,-0.1830127019) node (1)   {$\bullet$};
\draw (0.2588190451,0.6830127019) node (2)   {$\bullet$};
\draw (0.7071067812,-0.5) node (3)   {$\bullet$};
\draw (0,-1.2247448714) node (4)   {$\bullet$};
\draw (1) -- (2);
\draw (1) -- (3);
\draw (1) -- (4);
\draw (2) -- (3);
\draw (2) -- (4);
\draw (3) -- (4);
\draw[ultra thick] (-0.9659258263,-0.1830127019) -- (0.2588190451,0.6830127019);
\draw[ultra thick] (0.2588190451,0.6830127019) -- (0.7071067812,-0.5);
\end{tikzpicture} & 
\begin{tikzpicture}
\draw (-0.9659258263,-0.1830127019) node (1)   {$\bullet$};
\draw (0.2588190451,0.6830127019) node (2)   {$\bullet$};
\draw (0.7071067812,-0.5) node (3)   {$\bullet$};
\draw (0,-1.2247448714) node (4)   {$\bullet$};
\draw (1) -- (2);
\draw (1) -- (3);
\draw (1) -- (4);
\draw (2) -- (3);
\draw (2) -- (4);
\draw (3) -- (4);
\draw[ultra thick] (-0.9659258263,-0.1830127019) -- (0.2588190451,0.6830127019);
\draw[ultra thick] (0.7071067812,-0.5) -- (0,-1.2247448714);
\end{tikzpicture} 
& 
 \begin{tikzpicture}
\draw (-0.9659258263,-0.1830127019) node (1)   {$\bullet$};
\draw (0.2588190451,0.6830127019) node (2)   {$\bullet$};
\draw (0.7071067812,-0.5) node (3)   {$\bullet$};
\draw (0,-1.2247448714) node (4)   {$\bullet$};
\draw (1) -- (2);
\draw (1) -- (3);
\draw (1) -- (4);
\draw (2) -- (3);
\draw (2) -- (4);
\draw (3) -- (4);
\draw[ultra thick] (-0.9659258263,-0.1830127019) -- (0.2588190451,0.6830127019);
\draw[ultra thick] (-0.9659258263,-0.1830127019) -- (0,-1.2247448714);
\draw[ultra thick] (0.2588190451,0.6830127019) -- (0,-1.2247448714);
\end{tikzpicture} & 
\begin{tikzpicture}
\draw (-0.9659258263,-0.1830127019) node (1)   {$\bullet$};
\draw (0.2588190451,0.6830127019) node (2)   {$\bullet$};
\draw (0.7071067812,-0.5) node (3)   {$\bullet$};
\draw (0,-1.2247448714) node (4)   {$\bullet$};
\draw (1) -- (2);
\draw (1) -- (3);
\draw (1) -- (4);
\draw (2) -- (3);
\draw (2) -- (4);
\draw (3) -- (4);
\draw[ultra thick] (-0.9659258263,-0.1830127019) -- (0.2588190451,0.6830127019);
\draw[ultra thick] (0.2588190451,0.6830127019) -- (0.7071067812,-0.5);
\draw[ultra thick] (0.7071067812,-0.5) -- (0,-1.2247448714);
\end{tikzpicture} & 
\begin{tikzpicture}
\draw (-0.9659258263,-0.1830127019) node (1)   {$\bullet$};
\draw (0.2588190451,0.6830127019) node (2)   {$\bullet$};
\draw (0.7071067812,-0.5) node (3)   {$\bullet$};
\draw (0,-1.2247448714) node (4)   {$\bullet$};
\draw (1) -- (2);
\draw (1) -- (3);
\draw (1) -- (4);
\draw (2) -- (3);
\draw (2) -- (4);
\draw (3) -- (4);
\draw[ultra thick] (-0.9659258263,-0.1830127019) -- (0,-1.2247448714);
\draw[ultra thick] (0.2588190451,0.6830127019) -- (0,-1.2247448714);
\draw[ultra thick] (0.7071067812,-0.5) -- (0,-1.2247448714);
\end{tikzpicture}
&
 \begin{tikzpicture}
\draw (-0.9659258263,-0.1830127019) node (1)   {$\bullet$};
\draw (0.2588190451,0.6830127019) node (2)   {$\bullet$};
\draw (0.7071067812,-0.5) node (3)   {$\bullet$};
\draw (0,-1.2247448714) node (4)   {$\bullet$};
\draw (1) -- (2);
\draw (1) -- (3);
\draw (1) -- (4);
\draw (2) -- (3);
\draw (2) -- (4);
\draw (3) -- (4);
\draw[ultra thick] (-0.9659258263,-0.1830127019) -- (0.2588190451,0.6830127019);
\draw[ultra thick] (-0.9659258263,-0.1830127019) -- (0,-1.2247448714);
\draw[ultra thick] (0.2588190451,0.6830127019) -- (0.7071067812,-0.5);
\draw[ultra thick] (0.2588190451,0.6830127019) -- (0,-1.2247448714);
\end{tikzpicture} & 
 \begin{tikzpicture}
\draw (-0.9659258263,-0.1830127019) node (1)   {$\bullet$};
\draw (0.2588190451,0.6830127019) node (2)   {$\bullet$};
\draw (0.7071067812,-0.5) node (3)   {$\bullet$};
\draw (0,-1.2247448714) node (4)   {$\bullet$};
\draw (1) -- (2);
\draw (1) -- (3);
\draw (1) -- (4);
\draw (2) -- (3);
\draw (2) -- (4);
\draw (3) -- (4);
\draw[ultra thick] (-0.9659258263,-0.1830127019) -- (0.7071067812,-0.5);
\draw[ultra thick] (-0.9659258263,-0.1830127019) -- (0,-1.2247448714);
\draw[ultra thick] (0.2588190451,0.6830127019) -- (0.7071067812,-0.5);
\draw[ultra thick] (0.2588190451,0.6830127019) -- (0,-1.2247448714);
\end{tikzpicture} & 
 \begin{tikzpicture}
\draw (-0.9659258263,-0.1830127019) node (1)   {$\bullet$};
\draw (0.2588190451,0.6830127019) node (2)   {$\bullet$};
\draw (0.7071067812,-0.5) node (3)   {$\bullet$};
\draw (0,-1.2247448714) node (4)   {$\bullet$};
\draw (1) -- (2);
\draw (1) -- (3);
\draw (1) -- (4);
\draw (2) -- (3);
\draw (2) -- (4);
\draw (3) -- (4);
\draw[ultra thick] (-0.9659258263,-0.1830127019) -- (0.7071067812,-0.5);
\draw[ultra thick] (-0.9659258263,-0.1830127019) -- (0,-1.2247448714);
\draw[ultra thick] (0.2588190451,0.6830127019) -- (0.7071067812,-0.5);
\draw[ultra thick] (0.2588190451,0.6830127019) -- (0,-1.2247448714);
\draw[ultra thick] (0.7071067812,-0.5) -- (0,-1.2247448714);
\end{tikzpicture}
& 
 \begin{tikzpicture}
\draw (-0.9659258263,-0.1830127019) node (1)   {$\bullet$};
\draw (0.2588190451,0.6830127019) node (2)   {$\bullet$};
\draw (0.7071067812,-0.5) node (3)   {$\bullet$};
\draw (0,-1.2247448714) node (4)   {$\bullet$};
\draw (1) -- (2);
\draw (1) -- (3);
\draw (1) -- (4);
\draw (2) -- (3);
\draw (2) -- (4);
\draw (3) -- (4);
\draw[ultra thick] (-0.9659258263,-0.1830127019) -- (0.2588190451,0.6830127019);
\draw[ultra thick] (-0.9659258263,-0.1830127019) -- (0.7071067812,-0.5);
\draw[ultra thick] (-0.9659258263,-0.1830127019) -- (0,-1.2247448714);
\draw[ultra thick] (0.2588190451,0.6830127019) -- (0.7071067812,-0.5);
\draw[ultra thick] (0.2588190451,0.6830127019) -- (0,-1.2247448714);
\draw[ultra thick] (0.7071067812,-0.5) -- (0,-1.2247448714);
\end{tikzpicture}   
  & 
 \begin{tikzpicture}
\fill[color=red] (-0.9659258263,-0.1830127019) -- (0.7071067812,-0.5) -- (0,-1.2247448714) -- cycle;
\draw (-0.9659258263,-0.1830127019) node (1)   {$\bullet$};
\draw (0.2588190451,0.6830127019) node (2)   {$\bullet$};
\draw (0.7071067812,-0.5) node (3)   {$\bullet$};
\draw (0,-1.2247448714) node (4)   {$\bullet$};
\draw (1) -- (2);
\draw (1) -- (3);
\draw (1) -- (4);
\draw (2) -- (3);
\draw (2) -- (4);
\draw (3) -- (4);
\draw[ultra thick] (-0.9659258263,-0.1830127019) -- (0.7071067812,-0.5);
\draw[ultra thick] (-0.9659258263,-0.1830127019) -- (0,-1.2247448714);
\draw[ultra thick] (0.7071067812,-0.5) -- (0,-1.2247448714);
\end{tikzpicture} &
   \begin{tikzpicture}
\fill[color=red] (-0.9659258263,-0.1830127019) -- (0.7071067812,-0.5) -- (0,-1.2247448714) -- cycle;
\draw (-0.9659258263,-0.1830127019) node (1)   {$\bullet$};
\draw (0.2588190451,0.6830127019) node (2)   {$\bullet$};
\draw (0.7071067812,-0.5) node (3)   {$\bullet$};
\draw (0,-1.2247448714) node (4)   {$\bullet$};
\draw (1) -- (2);
\draw (1) -- (3);
\draw (1) -- (4);
\draw (2) -- (3);
\draw (2) -- (4);
\draw (3) -- (4);
\draw[ultra thick] (-0.9659258263,-0.1830127019) -- (0.2588190451,0.6830127019);
\draw[ultra thick] (-0.9659258263,-0.1830127019) -- (0.7071067812,-0.5);
\draw[ultra thick] (-0.9659258263,-0.1830127019) -- (0,-1.2247448714);
\draw[ultra thick] (0.7071067812,-0.5) -- (0,-1.2247448714);
\end{tikzpicture} & 
  \begin{tikzpicture}
\fill[color=red] (-0.9659258263,-0.1830127019) -- (0.7071067812,-0.5) -- (0,-1.2247448714) -- cycle;
\draw (-0.9659258263,-0.1830127019) node (1)   {$\bullet$};
\draw (0.2588190451,0.6830127019) node (2)   {$\bullet$};
\draw (0.7071067812,-0.5) node (3)   {$\bullet$};
\draw (0,-1.2247448714) node (4)   {$\bullet$};
\draw (1) -- (2);
\draw (1) -- (3);
\draw (1) -- (4);
\draw (2) -- (3);
\draw (2) -- (4);
\draw (3) -- (4);
\draw[ultra thick] (-0.9659258263,-0.1830127019) -- (0.2588190451,0.6830127019);
\draw[ultra thick] (-0.9659258263,-0.1830127019) -- (0.7071067812,-0.5);
\draw[ultra thick] (-0.9659258263,-0.1830127019) -- (0,-1.2247448714);
\draw[ultra thick] (0.2588190451,0.6830127019) -- (0,-1.2247448714);
\draw[ultra thick] (0.7071067812,-0.5) -- (0,-1.2247448714);
\end{tikzpicture}
&
 \begin{tikzpicture}
\fill[color=red] (-0.9659258263,-0.1830127019) -- (0.7071067812,-0.5) -- (0,-1.2247448714) -- cycle;
\draw (-0.9659258263,-0.1830127019) node (1)   {$\bullet$};
\draw (0.2588190451,0.6830127019) node (2)   {$\bullet$};
\draw (0.7071067812,-0.5) node (3)   {$\bullet$};
\draw (0,-1.2247448714) node (4)   {$\bullet$};
\draw (1) -- (2);
\draw (1) -- (3);
\draw (1) -- (4);
\draw (2) -- (3);
\draw (2) -- (4);
\draw (3) -- (4);
\draw[ultra thick] (-0.9659258263,-0.1830127019) -- (0.2588190451,0.6830127019);
\draw[ultra thick] (-0.9659258263,-0.1830127019) -- (0.7071067812,-0.5);
\draw[ultra thick] (-0.9659258263,-0.1830127019) -- (0,-1.2247448714);
\draw[ultra thick] (0.2588190451,0.6830127019) -- (0.7071067812,-0.5);
\draw[ultra thick] (0.2588190451,0.6830127019) -- (0,-1.2247448714);
\draw[ultra thick] (0.7071067812,-0.5) -- (0,-1.2247448714);
\end{tikzpicture}
&
 \begin{tikzpicture}
\draw[ultra thick] (-0.9659258263,-0.1830127019) -- (0.7071067812,-0.5);
\fill[color=red] (-0.9659258263,-0.1830127019) -- (0.2588190451,0.6830127019) -- (0,-1.2247448714) -- cycle;
\fill[color=red] (-0.9659258263,-0.1830127019) -- (0.7071067812,-0.5) -- (0,-1.2247448714) -- cycle;
\draw (-0.9659258263,-0.1830127019) node (1)   {$\bullet$};
\draw (0.2588190451,0.6830127019) node (2)   {$\bullet$};
\draw (0.7071067812,-0.5) node (3)   {$\bullet$};
\draw (0,-1.2247448714) node (4)   {$\bullet$};
\draw (1) -- (2);
\draw (1) -- (4);
\draw (2) -- (3);
\draw (2) -- (4);
\draw (3) -- (4);
\draw[ultra thick] (-0.9659258263,-0.1830127019) -- (0.2588190451,0.6830127019);
\draw[ultra thick] (-0.9659258263,-0.1830127019) -- (0,-1.2247448714);
\draw[ultra thick] (0.2588190451,0.6830127019) -- (0,-1.2247448714);
\draw[ultra thick] (0.7071067812,-0.5) -- (0,-1.2247448714);
\end{tikzpicture} & 
 \begin{tikzpicture}
\draw[ultra thick] (-0.9659258263,-0.1830127019) -- (0.7071067812,-0.5);
\fill[color=red] (-0.9659258263,-0.1830127019) -- (0.2588190451,0.6830127019) -- (0,-1.2247448714) -- cycle;
\fill[color=red] (-0.9659258263,-0.1830127019) -- (0.7071067812,-0.5) -- (0,-1.2247448714) -- cycle;
\draw (-0.9659258263,-0.1830127019) node (1)   {$\bullet$};
\draw (0.2588190451,0.6830127019) node (2)   {$\bullet$};
\draw (0.7071067812,-0.5) node (3)   {$\bullet$};
\draw (0,-1.2247448714) node (4)   {$\bullet$};
\draw (1) -- (2);
\draw (1) -- (4);
\draw (2) -- (3);
\draw (2) -- (4);
\draw (3) -- (4);
\draw[ultra thick] (-0.9659258263,-0.1830127019) -- (0.2588190451,0.6830127019);
\draw[ultra thick] (-0.9659258263,-0.1830127019) -- (0,-1.2247448714);
\draw[ultra thick] (0.2588190451,0.6830127019) -- (0.7071067812,-0.5);
\draw[ultra thick] (0.2588190451,0.6830127019) -- (0,-1.2247448714);
\draw[ultra thick] (0.7071067812,-0.5) -- (0,-1.2247448714);
\end{tikzpicture}
&
 \begin{tikzpicture}
\fill[color=red] (-0.9659258263,-0.1830127019) -- (0.2588190451,0.6830127019) -- (0.7071067812,-0.5) -- cycle;
\fill[color=red] (0.2588190451,0.6830127019) -- (0.7071067812,-0.5) -- (0,-1.2247448714) -- cycle;
\draw (-0.9659258263,-0.1830127019) node (1)   {$\bullet$};
\draw (0.2588190451,0.6830127019) node (2)   {$\bullet$};
\draw (0.7071067812,-0.5) node (3)   {$\bullet$};
\draw (0,-1.2247448714) node (4)   {$\bullet$};
\draw (1) -- (2);
\draw (1) -- (3);
\draw (1) -- (4);
\draw (2) -- (3);
\draw (2) -- (4);
\draw (3) -- (4);
\draw[ultra thick] (-0.9659258263,-0.1830127019) -- (0.7071067812,-0.5);
\fill[color=red] (-0.9659258263,-0.1830127019) -- (0.2588190451,0.6830127019) -- (0,-1.2247448714) -- cycle;
\draw[ultra thick] (-0.9659258263,-0.1830127019) -- (0.2588190451,0.6830127019);
\draw[ultra thick] (-0.9659258263,-0.1830127019) -- (0,-1.2247448714);
\draw[ultra thick] (0.2588190451,0.6830127019) -- (0.7071067812,-0.5);
\draw[ultra thick] (0.2588190451,0.6830127019) -- (0,-1.2247448714);
\draw[ultra thick] (0.7071067812,-0.5) -- (0,-1.2247448714);
\end{tikzpicture} & 
 \begin{tikzpicture}
\fill[color=red] (-0.9659258263,-0.1830127019) -- (0.2588190451,0.6830127019) -- (0.7071067812,-0.5) -- cycle;
\fill[color=red] (0.2588190451,0.6830127019) -- (0.7071067812,-0.5) -- (0,-1.2247448714) -- cycle;
\draw (-0.9659258263,-0.1830127019) node (1)   {$\bullet$};
\draw (0.2588190451,0.6830127019) node (2)   {$\bullet$};
\draw (0.7071067812,-0.5) node (3)   {$\bullet$};
\draw (0,-1.2247448714) node (4)   {$\bullet$};
\draw (1) -- (2);
\draw (1) -- (4);
\draw (2) -- (3);
\draw (2) -- (4);
\draw (3) -- (4);
\fill[color=red] (-0.9659258263,-0.1830127019) -- (0.2588190451,0.6830127019) -- (0,-1.2247448714) -- cycle;
\draw[ultra thick] (-0.9659258263,-0.1830127019) -- (0.2588190451,0.6830127019);
\draw[ultra thick] (-0.9659258263,-0.1830127019) -- (0,-1.2247448714);
\draw[ultra thick] (0.2588190451,0.6830127019) -- (0.7071067812,-0.5);
\draw[ultra thick] (0.2588190451,0.6830127019) -- (0,-1.2247448714);
\draw[ultra thick] (0.7071067812,-0.5) -- (0,-1.2247448714);
\draw[ultra thick,dashed] (-0.9659258263,-0.1830127019) -- (0.7071067812,-0.5);
\end{tikzpicture} &
 \begin{tikzpicture}
\fill[color=blue] (-0.9659258263,-0.1830127019) -- (0.2588190451,0.6830127019) -- (0.7071067812,-0.5) -- cycle;
\fill[color=blue] (0.2588190451,0.6830127019) -- (0.7071067812,-0.5) -- (0,-1.2247448714) -- cycle;
\draw (-0.9659258263,-0.1830127019) node (1)   {$\bullet$};
\draw (0.2588190451,0.6830127019) node (2)   {$\bullet$};
\draw (0.7071067812,-0.5) node (3)   {$\bullet$};
\draw (0,-1.2247448714) node (4)   {$\bullet$};
\draw (1) -- (2);
\draw (1) -- (4);
\draw (2) -- (3);
\draw (2) -- (4);
\draw (3) -- (4);
\fill[color=blue] (-0.9659258263,-0.1830127019) -- (0.2588190451,0.6830127019) -- (0,-1.2247448714) -- cycle;
\draw[ultra thick] (-0.9659258263,-0.1830127019) -- (0.2588190451,0.6830127019);
\draw[ultra thick] (-0.9659258263,-0.1830127019) -- (0,-1.2247448714);
\draw[ultra thick] (0.2588190451,0.6830127019) -- (0.7071067812,-0.5);
\draw[ultra thick] (0.2588190451,0.6830127019) -- (0,-1.2247448714);
\draw[ultra thick] (0.7071067812,-0.5) -- (0,-1.2247448714);
\draw[thick,dashed] (-0.9659258263,-0.1830127019) -- (0.7071067812,-0.5);
\end{tikzpicture} \\
\hline
\begin{tikzpicture}
\draw (-0.9659258263,-0.1830127019) node (1)   {$\bullet$};
\draw (0.2588190451,0.6830127019) node (2)   {$\bullet$};
\draw (0.7071067812,-0.5) node (3)   {$\bullet$};
\draw (0,-1.2247448714) node (4)   {$\bullet$};
\draw (1) -- (2);
\draw (1) -- (3);
\draw (1) -- (4);
\draw (2) -- (3);
\draw (2) -- (4);
\draw (3) -- (4);
\end{tikzpicture}  
 & \Huge \begin{array}{c} 0 \\ \ \\ \ \\ \end{array}    & 
\resizebox{1cm}{!}{$\frac{1}{2}$}   & 
\resizebox{1cm}{!}{$\frac{1}{2}$}   & 
\resizebox{1cm}{!}{$\frac{1}{2}$}   & 
\resizebox{1cm}{!}{$\frac{1}{2}$}   & 
\resizebox{1cm}{!}{$\frac{1}{2}$}   & 
\resizebox{1cm}{!}{$\frac{1}{2}$}   & 
\resizebox{1cm}{!}{$\frac{1}{2}$}   & 
\resizebox{1cm}{!}{$\frac{1}{2}$}   & 
\resizebox{1cm}{!}{$\frac{1}{2}$}   & 
\resizebox{1cm}{!}{$\frac{1}{2}$}   & 
\resizebox{1cm}{!}{$\frac{2}{3}$}   & 
\resizebox{1cm}{!}{$\frac{2}{3}$}   & 
\resizebox{1cm}{!}{$\frac{2}{3}$}   & 
\resizebox{1cm}{!}{$\frac{2}{3}$}   & 
\resizebox{1cm}{!}{$\frac{2}{3}$}   & 
\resizebox{1cm}{!}{$\frac{2}{3}$}   & 
\resizebox{1cm}{!}{$\frac{2}{3}$}   & 
\resizebox{1cm}{!}{$\frac{2}{3}$}   & 
\resizebox{1cm}{!}{$\frac{3}{4}$}  \\ 
\hline 
\begin{tikzpicture}
\draw (-0.9659258263,-0.1830127019) node (1)   {$\bullet$};
\draw (0.2588190451,0.6830127019) node (2)   {$\bullet$};
\draw (0.7071067812,-0.5) node (3)   {$\bullet$};
\draw (0,-1.2247448714) node (4)   {$\bullet$};
\draw (1) -- (2);
\draw (1) -- (3);
\draw (1) -- (4);
\draw (2) -- (3);
\draw (2) -- (4);
\draw (3) -- (4);
\draw[ultra thick] (-0.9659258263,-0.1830127019) -- (0.2588190451,0.6830127019);
\end{tikzpicture} &    & 
\Huge \begin{array}{c} 0 \\ \ \\ \ \\ \end{array}    & 
\resizebox{1cm}{!}{$\frac{1}{2}$}   & 
\resizebox{1cm}{!}{$\frac{1}{2}$}   & 
\resizebox{1cm}{!}{$\frac{1}{2}$}   & 
\resizebox{1cm}{!}{$\frac{1}{2}$}   & 
\resizebox{1cm}{!}{$\frac{1}{2}$}   & 
\resizebox{1cm}{!}{$\frac{1}{2}$}   & 
\resizebox{1cm}{!}{$\frac{1}{2}$}   & 
\resizebox{1cm}{!}{$\frac{1}{2}$}   & 
\resizebox{1cm}{!}{$\frac{1}{2}$}   & 
\resizebox{1cm}{!}{$\frac{1}{2}$}   & 
\resizebox{1cm}{!}{$\frac{1}{2}$}   & 
\resizebox{1cm}{!}{$\frac{1}{2}$}   & 
\resizebox{1cm}{!}{$\frac{1}{2}$}   & 
\resizebox{1cm}{!}{$\frac{1}{2}$}   & 
\resizebox{1cm}{!}{$\frac{1}{2}$}   & 
\resizebox{1cm}{!}{$\frac{2}{3}$}   & 
\resizebox{1cm}{!}{$\frac{2}{3}$}   & 
\resizebox{1cm}{!}{$\frac{2}{3}$}  \\ 
\hline 

\begin{tikzpicture}
\draw (-0.9659258263,-0.1830127019) node (1)   {$\bullet$};
\draw (0.2588190451,0.6830127019) node (2)   {$\bullet$};
\draw (0.7071067812,-0.5) node (3)   {$\bullet$};
\draw (0,-1.2247448714) node (4)   {$\bullet$};
\draw (1) -- (2);
\draw (1) -- (3);
\draw (1) -- (4);
\draw (2) -- (3);
\draw (2) -- (4);
\draw (3) -- (4);
\draw[ultra thick] (-0.9659258263,-0.1830127019) -- (0.2588190451,0.6830127019);
\draw[ultra thick] (0.2588190451,0.6830127019) -- (0.7071067812,-0.5);
\end{tikzpicture} &    & 
   & 
\Huge \begin{array}{c} 0 \\ \ \\ \ \\ \end{array}    & 
\resizebox{1cm}{!}{$\frac{1}{2}$}   & 
\resizebox{1cm}{!}{$\frac{1}{2}$}   & 
\resizebox{1cm}{!}{$\frac{1}{2}$}   & 
\resizebox{1cm}{!}{$\frac{1}{2}$}   & 
\resizebox{1cm}{!}{$\frac{1}{2}$}   & 
\resizebox{1cm}{!}{$\frac{1}{2}$}   & 
\resizebox{1cm}{!}{$\frac{1}{2}$}   & 
\resizebox{1cm}{!}{$\frac{1}{2}$}   & 
\resizebox{1cm}{!}{$\frac{1}{2}$}   & 
\resizebox{1cm}{!}{$\frac{1}{2}$}   & 
\resizebox{1cm}{!}{$\frac{1}{2}$}   & 
\resizebox{1cm}{!}{$\frac{1}{2}$}   & 
\resizebox{1cm}{!}{$\frac{1}{2}$}   & 
\resizebox{1cm}{!}{$\frac{1}{2}$}   & 
\resizebox{1cm}{!}{$\frac{1}{2}$}   & 
\resizebox{1cm}{!}{$\frac{2}{3}$}   & 
\resizebox{1cm}{!}{$\frac{2}{3}$}  \\ 
\hline 

\begin{tikzpicture}
\draw (-0.9659258263,-0.1830127019) node (1)   {$\bullet$};
\draw (0.2588190451,0.6830127019) node (2)   {$\bullet$};
\draw (0.7071067812,-0.5) node (3)   {$\bullet$};
\draw (0,-1.2247448714) node (4)   {$\bullet$};
\draw (1) -- (2);
\draw (1) -- (3);
\draw (1) -- (4);
\draw (2) -- (3);
\draw (2) -- (4);
\draw (3) -- (4);
\draw[ultra thick] (-0.9659258263,-0.1830127019) -- (0.2588190451,0.6830127019);
\draw[ultra thick] (0.7071067812,-0.5) -- (0,-1.2247448714);
\end{tikzpicture} &       & 
   & 
   & 
\Huge \begin{array}{c} 0 \\ \ \\ \ \\ \end{array}    & 
\resizebox{1cm}{!}{$\frac{1}{2}$}   & 
\resizebox{1cm}{!}{$\frac{1}{2}$}   & 
\resizebox{1cm}{!}{$\frac{1}{2}$}   & 
\resizebox{1cm}{!}{$\frac{1}{2}$}   & 
\resizebox{1cm}{!}{$\frac{1}{2}$}   & 
\resizebox{1cm}{!}{$\frac{1}{2}$}   & 
\resizebox{1cm}{!}{$\frac{1}{2}$}   & 
\resizebox{1cm}{!}{$\frac{1}{2}$}   & 
\resizebox{1cm}{!}{$\frac{1}{2}$}   & 
\resizebox{1cm}{!}{$\frac{1}{2}$}   & 
\resizebox{1cm}{!}{$\frac{1}{2}$}   & 
\resizebox{1cm}{!}{$\frac{1}{2}$}   & 
\resizebox{1cm}{!}{$\frac{1}{2}$}   & 
\resizebox{1cm}{!}{$\frac{1}{2}$}   & 
\resizebox{1cm}{!}{$\frac{1}{2}$}   & 
\resizebox{1cm}{!}{$\frac{1}{2}$}  \\ 
\hline 

\begin{tikzpicture}
\draw (-0.9659258263,-0.1830127019) node (1)   {$\bullet$};
\draw (0.2588190451,0.6830127019) node (2)   {$\bullet$};
\draw (0.7071067812,-0.5) node (3)   {$\bullet$};
\draw (0,-1.2247448714) node (4)   {$\bullet$};
\draw (1) -- (2);
\draw (1) -- (3);
\draw (1) -- (4);
\draw (2) -- (3);
\draw (2) -- (4);
\draw (3) -- (4);
\draw[ultra thick] (-0.9659258263,-0.1830127019) -- (0.2588190451,0.6830127019);
\draw[ultra thick] (-0.9659258263,-0.1830127019) -- (0,-1.2247448714);
\draw[ultra thick] (0.2588190451,0.6830127019) -- (0,-1.2247448714);
\end{tikzpicture}&    & 
   & 
   & 
   & 
\Huge \begin{array}{c} 0 \\ \ \\ \ \\ \end{array}    & 
\resizebox{1cm}{!}{$\frac{1}{2}$}   & 
\resizebox{1cm}{!}{$\frac{1}{2}$}   & 
\resizebox{1cm}{!}{$\frac{1}{2}$}   & 
\resizebox{1cm}{!}{$\frac{1}{2}$}   & 
\resizebox{1cm}{!}{$\frac{1}{2}$}   & 
\resizebox{1cm}{!}{$\frac{1}{2}$}   & 
\resizebox{1cm}{!}{$\frac{1}{3}$}   & 
\resizebox{1cm}{!}{$\frac{1}{2}$}   & 
\resizebox{1cm}{!}{$\frac{1}{2}$}   & 
\resizebox{1cm}{!}{$\frac{1}{2}$}   & 
\resizebox{1cm}{!}{$\frac{1}{2}$}   & 
\resizebox{1cm}{!}{$\frac{1}{2}$}   & 
\resizebox{1cm}{!}{$\frac{1}{2}$}   & 
\resizebox{1cm}{!}{$\frac{1}{2}$}   & 
\resizebox{1cm}{!}{$\frac{3}{5}$}  \\ 
\hline 
\begin{tikzpicture}
\draw (-0.9659258263,-0.1830127019) node (1)   {$\bullet$};
\draw (0.2588190451,0.6830127019) node (2)   {$\bullet$};
\draw (0.7071067812,-0.5) node (3)   {$\bullet$};
\draw (0,-1.2247448714) node (4)   {$\bullet$};
\draw (1) -- (2);
\draw (1) -- (3);
\draw (1) -- (4);
\draw (2) -- (3);
\draw (2) -- (4);
\draw (3) -- (4);
\draw[ultra thick] (-0.9659258263,-0.1830127019) -- (0.2588190451,0.6830127019);
\draw[ultra thick] (0.2588190451,0.6830127019) -- (0.7071067812,-0.5);
\draw[ultra thick] (0.7071067812,-0.5) -- (0,-1.2247448714);
\end{tikzpicture}&    & 
   & 
   & 
   & 
   & 
\Huge \begin{array}{c} 0 \\ \ \\ \ \\ \end{array}    & 
\resizebox{1cm}{!}{$\frac{1}{2}$}   & 
\resizebox{1cm}{!}{$\frac{1}{2}$}   & 
\resizebox{1cm}{!}{$\frac{1}{2}$}   & 
\resizebox{1cm}{!}{$\frac{1}{2}$}   & 
\resizebox{1cm}{!}{$\frac{1}{2}$}   & 
\resizebox{1cm}{!}{$\frac{1}{2}$}   & 
\resizebox{1cm}{!}{$\frac{1}{2}$}   & 
\resizebox{1cm}{!}{$\frac{1}{2}$}   & 
\resizebox{1cm}{!}{$\frac{1}{2}$}   & 
\resizebox{1cm}{!}{$\frac{1}{2}$}   & 
\resizebox{1cm}{!}{$\frac{1}{2}$}   & 
\resizebox{1cm}{!}{$\frac{1}{2}$}   & 
\resizebox{1cm}{!}{$\frac{1}{2}$}   & 
\resizebox{1cm}{!}{$\frac{1}{2}$}  \\ 
\hline 
\begin{tikzpicture}
\draw (-0.9659258263,-0.1830127019) node (1)   {$\bullet$};
\draw (0.2588190451,0.6830127019) node (2)   {$\bullet$};
\draw (0.7071067812,-0.5) node (3)   {$\bullet$};
\draw (0,-1.2247448714) node (4)   {$\bullet$};
\draw (1) -- (2);
\draw (1) -- (3);
\draw (1) -- (4);
\draw (2) -- (3);
\draw (2) -- (4);
\draw (3) -- (4);
\draw[ultra thick] (-0.9659258263,-0.1830127019) -- (0,-1.2247448714);
\draw[ultra thick] (0.2588190451,0.6830127019) -- (0,-1.2247448714);
\draw[ultra thick] (0.7071067812,-0.5) -- (0,-1.2247448714);
\end{tikzpicture} &    & 
   & 
   & 
   & 
   & 
   & 
\Huge \begin{array}{c} 0 \\ \ \\ \ \\ \end{array}    & 
\resizebox{1cm}{!}{$\frac{1}{2}$}   & 
\resizebox{1cm}{!}{$\frac{1}{2}$}   & 
\resizebox{1cm}{!}{$\frac{1}{2}$}   & 
\resizebox{1cm}{!}{$\frac{1}{2}$}   & 
\resizebox{1cm}{!}{$\frac{1}{2}$}   & 
\resizebox{1cm}{!}{$\frac{1}{2}$}   & 
\resizebox{1cm}{!}{$\frac{1}{2}$}   & 
\resizebox{1cm}{!}{$\frac{1}{2}$}   & 
\resizebox{1cm}{!}{$\frac{1}{2}$}   & 
\resizebox{1cm}{!}{$\frac{1}{2}$}   & 
\resizebox{1cm}{!}{$\frac{1}{2}$}   & 
\resizebox{1cm}{!}{$\frac{2}{3}$}   & 
\resizebox{1cm}{!}{$\frac{2}{3}$}  \\ 
\hline 
\begin{tikzpicture}
\draw (-0.9659258263,-0.1830127019) node (1)   {$\bullet$};
\draw (0.2588190451,0.6830127019) node (2)   {$\bullet$};
\draw (0.7071067812,-0.5) node (3)   {$\bullet$};
\draw (0,-1.2247448714) node (4)   {$\bullet$};
\draw (1) -- (2);
\draw (1) -- (3);
\draw (1) -- (4);
\draw (2) -- (3);
\draw (2) -- (4);
\draw (3) -- (4);
\draw[ultra thick] (-0.9659258263,-0.1830127019) -- (0.2588190451,0.6830127019);
\draw[ultra thick] (-0.9659258263,-0.1830127019) -- (0,-1.2247448714);
\draw[ultra thick] (0.2588190451,0.6830127019) -- (0.7071067812,-0.5);
\draw[ultra thick] (0.2588190451,0.6830127019) -- (0,-1.2247448714);
\end{tikzpicture}&    & 
   & 
   & 
   & 
   & 
   & 
   & 
\Huge \begin{array}{c} 0 \\ \ \\ \ \\ \end{array}    & 
\resizebox{1cm}{!}{$\frac{1}{2}$}   & 
\resizebox{1cm}{!}{$\frac{1}{2}$}   & 
\resizebox{1cm}{!}{$\frac{1}{2}$}   & 
\resizebox{1cm}{!}{$\frac{1}{2}$}   & 
\resizebox{1cm}{!}{$\frac{1}{3}$}   & 
\resizebox{1cm}{!}{$\frac{1}{2}$}   & 
\resizebox{1cm}{!}{$\frac{1}{2}$}   & 
\resizebox{1cm}{!}{$\frac{1}{2}$}   & 
\resizebox{1cm}{!}{$\frac{1}{2}$}   & 
\resizebox{1cm}{!}{$\frac{1}{2}$}   & 
\resizebox{1cm}{!}{$\frac{1}{2}$}   & 
\resizebox{1cm}{!}{$\frac{1}{2}$}  \\ 
\hline 
\begin{tikzpicture}
\draw (-0.9659258263,-0.1830127019) node (1)   {$\bullet$};
\draw (0.2588190451,0.6830127019) node (2)   {$\bullet$};
\draw (0.7071067812,-0.5) node (3)   {$\bullet$};
\draw (0,-1.2247448714) node (4)   {$\bullet$};
\draw (1) -- (2);
\draw (1) -- (3);
\draw (1) -- (4);
\draw (2) -- (3);
\draw (2) -- (4);
\draw (3) -- (4);
\draw[ultra thick] (-0.9659258263,-0.1830127019) -- (0.7071067812,-0.5);
\draw[ultra thick] (-0.9659258263,-0.1830127019) -- (0,-1.2247448714);
\draw[ultra thick] (0.2588190451,0.6830127019) -- (0.7071067812,-0.5);
\draw[ultra thick] (0.2588190451,0.6830127019) -- (0,-1.2247448714);
\end{tikzpicture} &    & 
   & 
   & 
   & 
   & 
   & 
   & 
   & 
\Huge \begin{array}{c} 0 \\ \ \\ \ \\ \end{array}    & 
\resizebox{1cm}{!}{$\frac{1}{2}$}   & 
\resizebox{1cm}{!}{$\frac{1}{2}$}   & 
\resizebox{1cm}{!}{$\frac{1}{2}$}   & 
\resizebox{1cm}{!}{$\frac{1}{2}$}   & 
\resizebox{1cm}{!}{$\frac{1}{2}$}   & 
\resizebox{1cm}{!}{$\frac{1}{2}$}   & 
\resizebox{1cm}{!}{$\frac{1}{2}$}   & 
\resizebox{1cm}{!}{$\frac{1}{2}$}   & 
\resizebox{1cm}{!}{$\frac{1}{2}$}   & 
\resizebox{1cm}{!}{$\frac{1}{2}$}   & 
\resizebox{1cm}{!}{$\frac{1}{2}$}  \\ 
\hline 
\begin{tikzpicture}
\draw (-0.9659258263,-0.1830127019) node (1)   {$\bullet$};
\draw (0.2588190451,0.6830127019) node (2)   {$\bullet$};
\draw (0.7071067812,-0.5) node (3)   {$\bullet$};
\draw (0,-1.2247448714) node (4)   {$\bullet$};
\draw (1) -- (2);
\draw (1) -- (3);
\draw (1) -- (4);
\draw (2) -- (3);
\draw (2) -- (4);
\draw (3) -- (4);
\draw[ultra thick] (-0.9659258263,-0.1830127019) -- (0.7071067812,-0.5);
\draw[ultra thick] (-0.9659258263,-0.1830127019) -- (0,-1.2247448714);
\draw[ultra thick] (0.2588190451,0.6830127019) -- (0.7071067812,-0.5);
\draw[ultra thick] (0.2588190451,0.6830127019) -- (0,-1.2247448714);
\draw[ultra thick] (0.7071067812,-0.5) -- (0,-1.2247448714);
\end{tikzpicture} &    & 
   & 
   & 
   & 
   & 
   & 
   & 
   & 
   & 
\Huge \begin{array}{c} 0 \\ \ \\ \ \\ \end{array}    & 
\resizebox{1cm}{!}{$\frac{1}{2}$}   & 
\resizebox{1cm}{!}{$\frac{1}{2}$}   & 
\resizebox{1cm}{!}{$\frac{1}{2}$}   & 
\resizebox{1cm}{!}{$\frac{1}{3}$}   & 
\resizebox{1cm}{!}{$\frac{1}{2}$}   & 
\resizebox{1cm}{!}{$\frac{1}{3}$}   & 
\resizebox{1cm}{!}{$\frac{1}{2}$}   & 
\resizebox{1cm}{!}{$\frac{1}{2}$}   & 
\resizebox{1cm}{!}{$\frac{1}{2}$}   & 
\resizebox{1cm}{!}{$\frac{1}{2}$}  \\ 
\hline 
\begin{tikzpicture}
\draw (-0.9659258263,-0.1830127019) node (1)   {$\bullet$};
\draw (0.2588190451,0.6830127019) node (2)   {$\bullet$};
\draw (0.7071067812,-0.5) node (3)   {$\bullet$};
\draw (0,-1.2247448714) node (4)   {$\bullet$};
\draw (1) -- (2);
\draw (1) -- (3);
\draw (1) -- (4);
\draw (2) -- (3);
\draw (2) -- (4);
\draw (3) -- (4);
\draw[ultra thick] (-0.9659258263,-0.1830127019) -- (0.2588190451,0.6830127019);
\draw[ultra thick] (-0.9659258263,-0.1830127019) -- (0.7071067812,-0.5);
\draw[ultra thick] (-0.9659258263,-0.1830127019) -- (0,-1.2247448714);
\draw[ultra thick] (0.2588190451,0.6830127019) -- (0.7071067812,-0.5);
\draw[ultra thick] (0.2588190451,0.6830127019) -- (0,-1.2247448714);
\draw[ultra thick] (0.7071067812,-0.5) -- (0,-1.2247448714);
\end{tikzpicture} &    & 
   & 
   & 
   & 
   & 
   & 
   & 
   & 
   & 
   & 
\Huge \begin{array}{c} 0 \\ \ \\ \ \\ \end{array}    & 
\resizebox{1cm}{!}{$\frac{1}{2}$}   & 
\resizebox{1cm}{!}{$\frac{1}{2}$}   & 
\resizebox{1cm}{!}{$\frac{1}{2}$}   & 
\resizebox{1cm}{!}{$\frac{1}{3}$}   & 
\resizebox{1cm}{!}{$\frac{1}{2}$}   & 
\resizebox{1cm}{!}{$\frac{1}{3}$}   & 
\resizebox{1cm}{!}{$\frac{1}{3}$}   & 
\resizebox{1cm}{!}{$\frac{1}{3}$}   & 
\resizebox{1cm}{!}{$\frac{1}{2}$}  \\ 
\hline 
\begin{tikzpicture}
\fill[color=red] (-0.9659258263,-0.1830127019) -- (0.7071067812,-0.5) -- (0,-1.2247448714) -- cycle;
\draw (-0.9659258263,-0.1830127019) node (1)   {$\bullet$};
\draw (0.2588190451,0.6830127019) node (2)   {$\bullet$};
\draw (0.7071067812,-0.5) node (3)   {$\bullet$};
\draw (0,-1.2247448714) node (4)   {$\bullet$};
\draw (1) -- (2);
\draw (1) -- (3);
\draw (1) -- (4);
\draw (2) -- (3);
\draw (2) -- (4);
\draw (3) -- (4);
\draw[ultra thick] (-0.9659258263,-0.1830127019) -- (0.7071067812,-0.5);
\draw[ultra thick] (-0.9659258263,-0.1830127019) -- (0,-1.2247448714);
\draw[ultra thick] (0.7071067812,-0.5) -- (0,-1.2247448714);
\end{tikzpicture} &    & 
   & 
   & 
   & 
   & 
   & 
   & 
   & 
   & 
   & 
   & 
\Huge \begin{array}{c} 0 \\ \ \\ \ \\ \end{array}    & 
\resizebox{1cm}{!}{$\frac{1}{2}$}   & 
\resizebox{1cm}{!}{$\frac{1}{2}$}   & 
\resizebox{1cm}{!}{$\frac{1}{2}$}   & 
\resizebox{1cm}{!}{$\frac{1}{2}$}   & 
\resizebox{1cm}{!}{$\frac{1}{2}$}   & 
\resizebox{1cm}{!}{$\frac{1}{2}$}   & 
\resizebox{1cm}{!}{$\frac{1}{2}$}   & 
\resizebox{1cm}{!}{$\frac{1}{2}$}  \\ 
\hline 
\begin{tikzpicture}
\fill[color=red] (-0.9659258263,-0.1830127019) -- (0.7071067812,-0.5) -- (0,-1.2247448714) -- cycle;
\draw (-0.9659258263,-0.1830127019) node (1)   {$\bullet$};
\draw (0.2588190451,0.6830127019) node (2)   {$\bullet$};
\draw (0.7071067812,-0.5) node (3)   {$\bullet$};
\draw (0,-1.2247448714) node (4)   {$\bullet$};
\draw (1) -- (2);
\draw (1) -- (3);
\draw (1) -- (4);
\draw (2) -- (3);
\draw (2) -- (4);
\draw (3) -- (4);
\draw[ultra thick] (-0.9659258263,-0.1830127019) -- (0.2588190451,0.6830127019);
\draw[ultra thick] (-0.9659258263,-0.1830127019) -- (0.7071067812,-0.5);
\draw[ultra thick] (-0.9659258263,-0.1830127019) -- (0,-1.2247448714);
\draw[ultra thick] (0.7071067812,-0.5) -- (0,-1.2247448714);
\end{tikzpicture} &    & 
   & 
   & 
   & 
   & 
   & 
   & 
   & 
   & 
   & 
   & 
   & 
\Huge \begin{array}{c} 0 \\ \ \\ \ \\ \end{array}    & 
\resizebox{1cm}{!}{$\frac{1}{2}$}   & 
\resizebox{1cm}{!}{$\frac{1}{2}$}   & 
\resizebox{1cm}{!}{$\frac{1}{2}$}   & 
\resizebox{1cm}{!}{$\frac{1}{2}$}   & 
\resizebox{1cm}{!}{$\frac{1}{2}$}   & 
\resizebox{1cm}{!}{$\frac{1}{2}$}   & 
\resizebox{1cm}{!}{$\frac{1}{2}$}  \\ 
\hline 
\begin{tikzpicture}
\fill[color=red] (-0.9659258263,-0.1830127019) -- (0.7071067812,-0.5) -- (0,-1.2247448714) -- cycle;
\draw (-0.9659258263,-0.1830127019) node (1)   {$\bullet$};
\draw (0.2588190451,0.6830127019) node (2)   {$\bullet$};
\draw (0.7071067812,-0.5) node (3)   {$\bullet$};
\draw (0,-1.2247448714) node (4)   {$\bullet$};
\draw (1) -- (2);
\draw (1) -- (3);
\draw (1) -- (4);
\draw (2) -- (3);
\draw (2) -- (4);
\draw (3) -- (4);
\draw[ultra thick] (-0.9659258263,-0.1830127019) -- (0.2588190451,0.6830127019);
\draw[ultra thick] (-0.9659258263,-0.1830127019) -- (0.7071067812,-0.5);
\draw[ultra thick] (-0.9659258263,-0.1830127019) -- (0,-1.2247448714);
\draw[ultra thick] (0.2588190451,0.6830127019) -- (0,-1.2247448714);
\draw[ultra thick] (0.7071067812,-0.5) -- (0,-1.2247448714);
\end{tikzpicture} &    & 
   & 
   & 
   & 
   & 
   & 
   & 
   & 
   & 
   & 
   & 
   & 
   & 
\Huge \begin{array}{c} 0 \\ \ \\ \ \\ \end{array}    & 
\resizebox{1cm}{!}{$\frac{1}{2}$}   & 
\resizebox{1cm}{!}{$\frac{1}{3}$}   & 
\resizebox{1cm}{!}{$\frac{1}{2}$}   & 
\resizebox{1cm}{!}{$\frac{1}{2}$}   & 
\resizebox{1cm}{!}{$\frac{1}{2}$}   & 
\resizebox{1cm}{!}{$\frac{1}{2}$}  \\ 
\hline 
\begin{tikzpicture}
\fill[color=red] (-0.9659258263,-0.1830127019) -- (0.7071067812,-0.5) -- (0,-1.2247448714) -- cycle;
\draw (-0.9659258263,-0.1830127019) node (1)   {$\bullet$};
\draw (0.2588190451,0.6830127019) node (2)   {$\bullet$};
\draw (0.7071067812,-0.5) node (3)   {$\bullet$};
\draw (0,-1.2247448714) node (4)   {$\bullet$};
\draw (1) -- (2);
\draw (1) -- (3);
\draw (1) -- (4);
\draw (2) -- (3);
\draw (2) -- (4);
\draw (3) -- (4);
\draw[ultra thick] (-0.9659258263,-0.1830127019) -- (0.2588190451,0.6830127019);
\draw[ultra thick] (-0.9659258263,-0.1830127019) -- (0.7071067812,-0.5);
\draw[ultra thick] (-0.9659258263,-0.1830127019) -- (0,-1.2247448714);
\draw[ultra thick] (0.2588190451,0.6830127019) -- (0.7071067812,-0.5);
\draw[ultra thick] (0.2588190451,0.6830127019) -- (0,-1.2247448714);
\draw[ultra thick] (0.7071067812,-0.5) -- (0,-1.2247448714);
\end{tikzpicture} &    & 
   & 
   & 
   & 
   & 
   & 
   & 
   & 
   & 
   & 
   & 
   & 
   & 
   & 
\Huge \begin{array}{c} 0 \\ \ \\ \ \\ \end{array}    & 
\resizebox{1cm}{!}{$\frac{1}{2}$}   & 
\resizebox{1cm}{!}{$\frac{1}{3}$}   & 
\resizebox{1cm}{!}{$\frac{1}{3}$}   & 
\resizebox{1cm}{!}{$\frac{1}{3}$}   & 
\resizebox{1cm}{!}{$\frac{2}{5}$}  \\ 
\hline 
\begin{tikzpicture}
\draw[ultra thick] (-0.9659258263,-0.1830127019) -- (0.7071067812,-0.5);
\fill[color=red] (-0.9659258263,-0.1830127019) -- (0.2588190451,0.6830127019) -- (0,-1.2247448714) -- cycle;
\fill[color=red] (-0.9659258263,-0.1830127019) -- (0.7071067812,-0.5) -- (0,-1.2247448714) -- cycle;
\draw (-0.9659258263,-0.1830127019) node (1)   {$\bullet$};
\draw (0.2588190451,0.6830127019) node (2)   {$\bullet$};
\draw (0.7071067812,-0.5) node (3)   {$\bullet$};
\draw (0,-1.2247448714) node (4)   {$\bullet$};
\draw (1) -- (2);
\draw (1) -- (4);
\draw (2) -- (3);
\draw (2) -- (4);
\draw (3) -- (4);
\draw[ultra thick] (-0.9659258263,-0.1830127019) -- (0.2588190451,0.6830127019);
\draw[ultra thick] (-0.9659258263,-0.1830127019) -- (0,-1.2247448714);
\draw[ultra thick] (0.2588190451,0.6830127019) -- (0,-1.2247448714);
\draw[ultra thick] (0.7071067812,-0.5) -- (0,-1.2247448714);
\end{tikzpicture} &    & 
   & 
   & 
   & 
   & 
   & 
   & 
   & 
   & 
   & 
   & 
   & 
   & 
   & 
   & 
\Huge \begin{array}{c} 0 \\ \ \\ \ \\ \end{array}    & 
\resizebox{1cm}{!}{$\frac{1}{2}$}   & 
\resizebox{1cm}{!}{$\frac{1}{2}$}   & 
\resizebox{1cm}{!}{$\frac{1}{2}$}   & 
\resizebox{1cm}{!}{$\frac{1}{2}$}  \\ 
\hline 
\begin{tikzpicture}
\draw[ultra thick] (-0.9659258263,-0.1830127019) -- (0.7071067812,-0.5);
\fill[color=red] (-0.9659258263,-0.1830127019) -- (0.2588190451,0.6830127019) -- (0,-1.2247448714) -- cycle;
\fill[color=red] (-0.9659258263,-0.1830127019) -- (0.7071067812,-0.5) -- (0,-1.2247448714) -- cycle;
\draw (-0.9659258263,-0.1830127019) node (1)   {$\bullet$};
\draw (0.2588190451,0.6830127019) node (2)   {$\bullet$};
\draw (0.7071067812,-0.5) node (3)   {$\bullet$};
\draw (0,-1.2247448714) node (4)   {$\bullet$};
\draw (1) -- (2);
\draw (1) -- (4);
\draw (2) -- (3);
\draw (2) -- (4);
\draw (3) -- (4);
\draw[ultra thick] (-0.9659258263,-0.1830127019) -- (0.2588190451,0.6830127019);
\draw[ultra thick] (-0.9659258263,-0.1830127019) -- (0,-1.2247448714);
\draw[ultra thick] (0.2588190451,0.6830127019) -- (0.7071067812,-0.5);
\draw[ultra thick] (0.2588190451,0.6830127019) -- (0,-1.2247448714);
\draw[ultra thick] (0.7071067812,-0.5) -- (0,-1.2247448714);
\end{tikzpicture} &    & 
   & 
   & 
   & 
   & 
   & 
   & 
   & 
   & 
   & 
   & 
   & 
   & 
   & 
   & 
   & 
\Huge \begin{array}{c} 0 \\ \ \\ \ \\ \end{array}    & 
\resizebox{1cm}{!}{$\frac{1}{3}$}   & 
\resizebox{1cm}{!}{$\frac{1}{3}$}   & 
\resizebox{1cm}{!}{$\frac{1}{3}$}  \\ 
\hline 
\begin{tikzpicture}
\fill[color=red] (-0.9659258263,-0.1830127019) -- (0.2588190451,0.6830127019) -- (0.7071067812,-0.5) -- cycle;
\fill[color=red] (0.2588190451,0.6830127019) -- (0.7071067812,-0.5) -- (0,-1.2247448714) -- cycle;
\draw (-0.9659258263,-0.1830127019) node (1)   {$\bullet$};
\draw (0.2588190451,0.6830127019) node (2)   {$\bullet$};
\draw (0.7071067812,-0.5) node (3)   {$\bullet$};
\draw (0,-1.2247448714) node (4)   {$\bullet$};
\draw (1) -- (2);
\draw (1) -- (3);
\draw (1) -- (4);
\draw (2) -- (3);
\draw (2) -- (4);
\draw (3) -- (4);
\draw[ultra thick] (-0.9659258263,-0.1830127019) -- (0.7071067812,-0.5);
\fill[color=red] (-0.9659258263,-0.1830127019) -- (0.2588190451,0.6830127019) -- (0,-1.2247448714) -- cycle;
\draw[ultra thick] (-0.9659258263,-0.1830127019) -- (0.2588190451,0.6830127019);
\draw[ultra thick] (-0.9659258263,-0.1830127019) -- (0,-1.2247448714);
\draw[ultra thick] (0.2588190451,0.6830127019) -- (0.7071067812,-0.5);
\draw[ultra thick] (0.2588190451,0.6830127019) -- (0,-1.2247448714);
\draw[ultra thick] (0.7071067812,-0.5) -- (0,-1.2247448714);
\end{tikzpicture} &    & 
   & 
   & 
   & 
   & 
   & 
   & 
   & 
   & 
   & 
   & 
   & 
   & 
   & 
   & 
   & 
   & 
\Huge \begin{array}{c} 0 \\ \ \\ \ \\ \end{array}    & 
\resizebox{1cm}{!}{$\frac{1}{3}$}   & 
\resizebox{1cm}{!}{$\frac{1}{3}$}  \\ 
\hline 
\begin{tikzpicture}
\fill[color=red] (-0.9659258263,-0.1830127019) -- (0.2588190451,0.6830127019) -- (0.7071067812,-0.5) -- cycle;
\fill[color=red] (0.2588190451,0.6830127019) -- (0.7071067812,-0.5) -- (0,-1.2247448714) -- cycle;
\draw (-0.9659258263,-0.1830127019) node (1)   {$\bullet$};
\draw (0.2588190451,0.6830127019) node (2)   {$\bullet$};
\draw (0.7071067812,-0.5) node (3)   {$\bullet$};
\draw (0,-1.2247448714) node (4)   {$\bullet$};
\draw (1) -- (2);
\draw (1) -- (4);
\draw (2) -- (3);
\draw (2) -- (4);
\draw (3) -- (4);
\fill[color=red] (-0.9659258263,-0.1830127019) -- (0.2588190451,0.6830127019) -- (0,-1.2247448714) -- cycle;
\draw[ultra thick] (-0.9659258263,-0.1830127019) -- (0.2588190451,0.6830127019);
\draw[ultra thick] (-0.9659258263,-0.1830127019) -- (0,-1.2247448714);
\draw[ultra thick] (0.2588190451,0.6830127019) -- (0.7071067812,-0.5);
\draw[ultra thick] (0.2588190451,0.6830127019) -- (0,-1.2247448714);
\draw[ultra thick] (0.7071067812,-0.5) -- (0,-1.2247448714);
\draw[ultra thick,dashed] (-0.9659258263,-0.1830127019) -- (0.7071067812,-0.5);
\end{tikzpicture} &    & 
   & 
   & 
   & 
   & 
   & 
   & 
   & 
   & 
   & 
   & 
   & 
   & 
   & 
   & 
   & 
   & 
   & 
\Huge \begin{array}{c} 0 \\ \ \\ \ \\ \end{array}    & 
\resizebox{1cm}{!}{$\frac{1}{4}$}
 \\ 
\hline 
\begin{tikzpicture}
\fill[color=blue] (-0.9659258263,-0.1830127019) -- (0.2588190451,0.6830127019) -- (0.7071067812,-0.5) -- cycle;
\fill[color=blue] (0.2588190451,0.6830127019) -- (0.7071067812,-0.5) -- (0,-1.2247448714) -- cycle;
\draw (-0.9659258263,-0.1830127019) node (1)   {$\bullet$};
\draw (0.2588190451,0.6830127019) node (2)   {$\bullet$};
\draw (0.7071067812,-0.5) node (3)   {$\bullet$};
\draw (0,-1.2247448714) node (4)   {$\bullet$};
\draw (1) -- (2);
\draw (1) -- (4);
\draw (2) -- (3);
\draw (2) -- (4);
\draw (3) -- (4);
\fill[color=blue] (-0.9659258263,-0.1830127019) -- (0.2588190451,0.6830127019) -- (0,-1.2247448714) -- cycle;
\draw[ultra thick] (-0.9659258263,-0.1830127019) -- (0.2588190451,0.6830127019);
\draw[ultra thick] (-0.9659258263,-0.1830127019) -- (0,-1.2247448714);
\draw[ultra thick] (0.2588190451,0.6830127019) -- (0.7071067812,-0.5);
\draw[ultra thick] (0.2588190451,0.6830127019) -- (0,-1.2247448714);
\draw[ultra thick] (0.7071067812,-0.5) -- (0,-1.2247448714);
\draw[thick,dashed] (-0.9659258263,-0.1830127019) -- (0.7071067812,-0.5);
\end{tikzpicture} &    & 
   & 
   & 
   & 
   & 
   & 
   & 
   & 
   & 
   & 
   & 
   & 
   & 
   & 
   & 
   & 
   & 
   & 
   & 
\Huge \begin{array}{c} 0 \\ \ \\ \ \\ \end{array}   \\ 
\hline
\end{array}
$
}
\caption{Simplicial complexes on 4 vertices}
\label{tab:simp4}
\end{table}

\end{document}